\documentclass[10pt,a4paper,oneside,titlepage=false]{amsart}
\usepackage{fullpage}
\usepackage{amsaddr}
\usepackage[style=numeric,doi=false, url=false,sortcites=true,maxnames=99,giveninits=true]{biblatex}
\RequirePackage{amsthm,amsmath,amsfonts,amssymb}

\theoremstyle{plain}
\newtheorem{axiom}{Axiom}
\newtheorem{claim}[axiom]{Claim}
\newtheorem{theorem}{Theorem}[section]
\newtheorem{lemma}[theorem]{Lemma}
\newtheorem{corollary}[theorem]{Corollary}

\theoremstyle{remark}
\newtheorem{definition}[theorem]{Definition}

\usepackage{pgfplots}
\usepackage{tikz}
\usetikzlibrary{shapes,decorations,arrows,calc,arrows.meta,fit,positioning}
\usepgfplotslibrary{fillbetween}
\usepackage{graphicx}
\usetikzlibrary{decorations.pathreplacing, matrix}
\usepackage{tabularx}

\usepackage[boxed, linesnumbered, noend, noline]{algorithm2e}
\SetKwInput{KwData}{Input}
\SetKwInput{KwResult}{Output}

\usepackage{tikz}
\usetikzlibrary{shapes,decorations,arrows,calc,arrows.meta,fit,positioning}
\usetikzlibrary{shadows,shadings,shapes.symbols, patterns}
\tikzset{
    double color fill/.code 2 args={
        \pgfdeclareverticalshading[%
            tikz@axis@top,tikz@axis@middle,tikz@axis@bottom%
        ]{diagonalfill}{100bp}{%
            color(0bp)=(tikz@axis@bottom);
            color(50bp)=(tikz@axis@bottom);
            color(50bp)=(tikz@axis@middle);
            color(50bp)=(tikz@axis@top);
            color(100bp)=(tikz@axis@top)
        }
        \tikzset{shade, left color=#1, right color=#2, shading=diagonalfill}
    }
}

\newcommand{\mdjackov}{m_{\mathrm{QGT}}}

\renewcommand\log{\ln}

\newcommand{\fB}{\mathfrak B}

\newcommand{\vX}{\vec X}

\renewcommand{\epsilon}{\eps}

\newcommand\vU{\vec U}

\newcommand\vH{\vec H}

\renewcommand{\vec}[1]{\boldsymbol{#1}}

\newcommand\SIGMA{\vec\sigma}

\newcommand\fD{\mathfrak{D}}

\newcommand\cE{\mathcal{E}}
\newcommand\cU{\mathcal{U}}
\newcommand\cN{\mathcal{N}}

\def\cE{{\mathcal E}}

\newcommand\eps{\varepsilon}

\newcommand\NN{\mathbb{N}}

\newcommand\Var{\mathrm{Var}}
\newcommand\Erw{\mathbb{E}}
\newcommand{\vecone}{\vec{1}}

\newcommand{\Bin}{{\rm Bin}}

\newcommand{\Hyp}{{\rm Hyp}}

\newcommand\bc[1]{\left({#1}\right)}
\newcommand\cbc[1]{\left\{{#1}\right\}}
\newcommand\bcfr[2]{\bc{\frac{#1}{#2}}}

\newcommand\brk[1]{\left\lbrack{#1}\right\rbrack}

\newcommand\abs[1]{\left|{#1}\right|}

\newcommand\RR{\mathbb{R}}

\newcommand\pr{\mathbb{P}} 
\renewcommand\Pr{\pr} 

\newcommand\Lem{Lemma}
\newcommand\Prop{Proposition}

\newcommand{\ceil}[1]{\left\lceil#1\right\rceil}

\def\pr{{\mathbb P}}

\newcommand{\remove}[1]{}

\DeclareMathOperator {\myOp}{SCIENT}
\newcommand{\algoname}{\ensuremath{\myOp}}

\pgfplotsset{compat=1.14}

\usepackage{float}

\title[Optimal Decoding from Sparse Pooled Data]{On a Near-Optimal \& Efficient Algorithm for the Sparse Pooled Data Problem}

\author{Max Hahn-Klimroth$^1$, Remco van der Hofstad$^2$, \\ Noela Müller$^2$, Connor Riddlesden$^2$}
\address{{\tt hahnklim@mathematik.uni-frankfurt.de}, \\ {\tt \{r.w.v.d.hofstad,n.s.muller, c.d.riddlesden \}@tue.nl }}
\address{$^1$Goethe-University Frankfurt \\ $^2$Eindhoven University of Technology, Department of Mathematics and Computer Science}

\begin{document}

\begin{abstract}
 The pooled data problem asks to identify the unknown labels of a set of items from condensed measurements. More precisely, given $n$ items, assume that each item has a label in $\cbc{0,1,\ldots, d}$, encoded via the ground-truth $\SIGMA$. We call the pooled data problem sparse if the number of non-zero entries of $\SIGMA$ scales as $k \sim n^{\theta}$ for $\theta \in (0,1)$. The information that is revealed about $\SIGMA$ comes from pooled measurements, each indicating how many items of each label are contained in the pool. The most basic question is to design a pooling scheme that uses as few pools as possible, while reconstructing $\SIGMA$ with high probability. Variants of the problem and its combinatorial ramifications have been studied for at least 35 years. However, the study of the modern question of \emph{efficient} inference of the labels has suggested a statistical-to-computational gap of order $\log n$ in the minimum number of pools needed for theoretically possible versus efficient inference. In this article, we resolve the question whether this $\log n$-gap is artificial or of a fundamental nature by the design of an efficient algorithm, called \algoname, based upon a novel pooling scheme on a number of pools very close to the information-theoretic threshold. 
\end{abstract}

\maketitle

\section{Introduction}

Consider the basic problem of learning the unknown labels of a set of items.  These labels might be distinguishing features like an age group or a class label in a machine learning task. In the \emph{pooled data problem}, the information that is accessible about the labels comes from pooled measurements, where each pool reveals \emph{how many items of each label} are contained in it \cite{alaoui_2017}. The main task is then to infer all labels using as few pools as possible. The probably most famous special case of the pooled data problem is the \emph{non-adaptive quantitative group testing problem (QGT)}, where only labels $0$ and $1$ are present  \cite{djackov_1975, gebhard2021quantitative, gebrinski_2000, karimi_2019_2,Soleymani_2023}. In the QGT problem, the default interpretation of the labels is presence or absence of a virus, and the task becomes to identify the infected individuals. Other practical applications of the pooled data problem include DNA screening \cite{sham_2002}, traffic monitoring \cite{wang_2015}, machine learning \cite{liang2021neural, martins_2014} and signal recovery \cite{mazumdar2022support}.

In this article, we assume the existence of a single dominant label. In related problems such as the QGT \cite{AJS_book} or the compressed sensing problems \cite{donoho_l1}, this corresponds to the assumption of sparsity and is a distinguishing feature of the theoretical analysis of the model. This \emph{sublinear} variant finds its applications in the context of epidemiology, as early numbers of defectives can be captured in this setting according to Heaps' Law. Recently, this version of the pooled data problem has also found applications in image moderation tasks, where the goal is to detect rare but inappropriate images \cite{liang2021neural}.

Within the realm of sparse label vectors, we are interested in an \emph{average-case analysis}. This means that we assume that the true label vector is chosen uniformly at random among all vectors having a prescribed number of labels of each type, which is known to us. In practical settings, the distributional assumption of uniformity might be realised by a random permutation of the items prior to inference.

Additionally, throughout the article, we require that all measurements are conducted in parallel, which corresponds to the \emph{non-adaptive} version of the problem. This assumption is predominant in theoretical work on the quantitative group testing problem  \cite{alaoui_2017, feige2020quantitative, scarlett_2017} and of relevance in practical applications due to scalability and stability considerations \cite{NIPS2014_fb8feff2}.

Specifically, a key question for such a uniformly chosen sparse label vector subject to non-adaptive measurements is: What is the minimum number of pools that are necessary to recover \emph{all} out of $n$ labels of a uniformly chosen label vector correctly with high probability\footnote{With high probability (w.h.p.) means with probability tending to $1$ as $n \to \infty$.}? This question can be studied from an information-theoretic as well as from an algorithmic perspective.
We call a sequence $m_{\mathrm{lbd}} = m_{\mathrm{lbd}}(n)$ of positive numbers an \textit{information-theoretic lower bound} if for any $\delta >0$ and any design on $m \leq (1-\delta) m_{\mathrm{lbd}}$ pools, the probability of making an error converges to one. On the other hand, a sequence $m_{\mathrm{ubd}} = m_{\mathrm{ubd}}(n)$ of positive numbers is called an \textit{information-theoretic upper bound}, if for every $\delta >0$ there exists a pool design on $m \geq (1+\delta) m_{\mathrm{ubd}}$ pools that allows exact recovery of the label vector w.h.p. (neglecting computational issues).  The information-theoretic version of the question about the minimum number of tests has been resolved completely rather recently. In particular, it has been shown that both the sublinear as well as the linear variants of the non-adaptive pooled data problem with a bounded number of labels display a sharp threshold behaviour, in the sense that the optimal information-theoretic upper and lower bounds coincide \cite{alaoui_2017,  feige2020quantitative, gebhard2021quantitative, scarlett_2017}. 

Having identified the precise number of pools with which inference is theoretically possible, the algorithmic question on the minimal number of pools for which the true label vector can be recovered \emph{efficiently} (meaning in polynomial time) can be addressed systematically. Such algorithmic questions are of high present interest in many high-dimensional inference problems. 
For sublinear the pooled data problem, our main result closes the previously existing $\ln n$-gap between the information-theoretic threshold and the best existing efficient pooling schemes. Indeed, our pooling scheme requires only a constant factor more measurements than information-theoretically necessary \cite{djackov_1975}. Moreover, the additional constant, which depends on the sparsity $\theta$, approaches one as $\theta \downarrow 0$. This result is enabled by a technically novel polynomial-time construction of a pooling scheme, which is based on the \emph{spatial coupling technique} from coding theory, in combination with a one-stage thresholding algorithm called \algoname.

\subsection{Model} \label{sec:model}
We now give the necessary details of our model. Throughout the article, the number of items will be denoted by $n$. Each of the $n$ items $x_1, \ldots, x_n$, is assigned a label $\SIGMA_i:=\SIGMA(x_i) \in \cbc{0, 1, 2, \ldots, d}$. The vector $\SIGMA$ is called the \emph{ground-truth}, and while $n$ will tend to infinity below, $d$ is assumed to be a fixed number. 
Denote by $k_0, \ldots, k_d$ the numbers of items with label $0, \ldots, d$ such that $k_i = \lfloor n^{\theta_i} \rfloor$ for some $\theta_i \in (0,1)$, $i=1, \ldots, d$. We set $k = \max\{k_i: i=1, \ldots, d\}$ and $\theta = \max\{\theta_i: i=1, \ldots, d\}$. 
Throughout the article, we restrict to the sublinear regime in which one of the $n$ labels, say $0$, is asymptotically dominant in the sense that $k = \lfloor n^{\theta}\rfloor$ for $\theta \in (0,1)$. The constant $\theta$ is also called the sparsity of the model. Given $k_1, \ldots, k_d$, we assume that $\SIGMA$ is uniformly distributed over all vectors in $\{0, \ldots, d\}^n$ that have $k_i$ coordinates of label $i$ for $i=1, \ldots, d$ and $n-k$ entries of label $0$.

A pooling scheme for $x_1, \ldots, x_n$ corresponds to the specification of $m$ pools $a_1, \ldots, a_m$, where each $a_i$ is a subset of $\{x_1, \ldots, x_n\}$ of arbitrary cardinality. 
We visualise pooling schemes $\vec G$ through bipartite graphs (see Figure \ref{fig:example}). In this graphical representation, the $n$ items constitute one set of vertices (the so-called \emph{variable nodes}), and the $m$ pools $a_1, \ldots, a_m$ constitute the second set of vertices (the so-called \emph{factor nodes}). An edge between a variable node $x_i$ and a factor node $a_j$ indicates that item $x_i$ is part of pool $a_j$. We also impose the condition that each item $x_i$ can only be a part of a specific measurement $a_j$ \emph{once}, thus $\vec G$ is a simple graph. Additionally, every factor node comes with a label vector, which encodes the histogram of the label frequencies within the associated pool (see Figure \ref{fig:example}).

The vector of measurement results of a given pooling scheme $\vec G$ will be denoted by $\hat \SIGMA_{\vec G}$. Recall that we restrict ourselves to non-adaptive pooling designs, where all measurements are conducted in parallel.

\begin{figure}[ht]
\centering
\begin{tikzpicture}[scale=0.65]
\node[circle, draw, minimum width=0.66cm] (x0) at (0, 0) {$1$};
\node[circle, draw, minimum width=0.66cm] (x1) at (2,0) {$2$};
\node[circle, draw, minimum width=0.66cm] (x2) at (4, 0) {$0$};
\node[circle, draw, minimum width=0.66cm] (x3) at (6, 0) {$0$};
\node[circle, draw, minimum width=0.66cm] (x4) at (8, 0) {$1$}; 
\node[circle, draw, minimum width=0.66cm] (x5) at (10, 0) {$0$};
\node[circle, draw, minimum width=0.66cm] (x6) at (12, 0) {$0$};

\node[rectangle, draw, minimum width=0.5cm, minimum height=0.5cm] (a1) at (0, -2.5) {$(1,1,1)$};
\node[rectangle, draw, minimum width=0.5cm, minimum height=0.5cm] (a2) at (3,-2.5) {$(2,2,0)$};
\node[rectangle, draw, minimum width=0.5cm, minimum height=0.5cm] (a3) at (6, -2.5) {$(3,1,1)$};
\node[rectangle, draw, minimum width=0.5cm, minimum height=0.5cm] (a4) at (9, -2.5) {$(3,1,0)$};
\node[rectangle, draw, minimum width=0.5cm, minimum height=0.5cm] (a5) at (12, -2.5) {$(2,1,0)$};

\path[draw] (x0) -- (a1);
\path[draw] (x0) -- (a2);
\path[draw] (x0) -- (a3);
\path[draw] (x1) -- (a1);
\path[draw] (x1) -- (a3);
\path[draw] (x2) -- (a1);
\path[draw] (x2) -- (a3);
\path[draw] (x2) -- (a2);
\path[draw] (x3) -- (a3);
\path[draw] (x3) -- (a4);
\path[draw] (x3) -- (a5);
\path[draw] (x4) -- (a2);
\path[draw] (x4) -- (a5);
\path[draw] (x4) -- (a4);
\path[draw] (x5) -- (a2);
\path[draw] (x5)-- (a4);
\path[draw] (x6) -- (a3);
\path[draw] (x6) -- (a5);
\path[draw] (x6) -- (a4);
\end{tikzpicture}
\caption{Graphical representation of a pooling scheme for $d=2$: Here, the $n=7$ items are represented by circles, while the $m=5$ pools are represented by rectangles. Edges between items and pools are present whenever an item is an element of the corresponding pool. The label of an item represents its actual label, while the label of a pool represents the output of its measurement.}
\label{fig:example}
\end{figure}
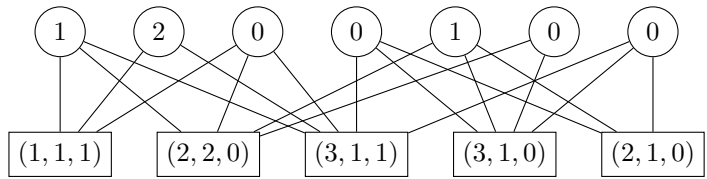
\subsection{Previous work}
This subsection provides a short overview over the literature that is closest to our setting. For this, recall that $k=\lfloor n^{\theta}\rfloor $ with $\theta \in (0,1)$. 
The combinatorial version of the quantitative group testing problem, where to goal is to correctly identify the ground truth with probability exactly \emph{equal} to one, has been studied since the 1960s \cite{shapiro_1960}. For this problem, \cite{djackov_1975} provides the lower bound 
\[m_{\mathrm{Djackov}} = (2+o(1)) \frac{1 - \theta}{ \theta} k\]
on the number of pools. 
Since learning $d+1$ labels is clearly harder than learning $2$ labels, $m_{\mathrm{Djackov}}$ also constitutes a lower bound for the general pooled data problem with $d+1$ labels.

For the general pooled data problem, and again in the combinatorial setting, the (non-constructive) arguments of  \cite{gebrinski_2000} yield the upper bound
\[m_{\mathrm{GK}}  = (4+o(1)) \frac{1 - \theta}{ \theta} k\]
on the number of necessary pools. This naturally also forms an upper bound on the number of pools that are necessary to infer the ground truth with high probability.

Recently, two independent works \cite{feige2020quantitative, gebhard2021quantitative} have shown that in the probabilistic quantitative group testing problem, there indeed is a sharp information-theoretic threshold at
\begin{align}\label{minf}
   \mdjackov^{\mathrm{inf}} =  2 \frac{1 - \theta}{ \theta} k.
\end{align}
This means that for any $\delta>0$, there exists no pooling scheme on $(1-\delta)\mdjackov^{\mathrm{inf}}$ tests that recovers the ground truth with non-vanishing probability, while there is a pooling scheme on $(1+\delta)\mdjackov^{\mathrm{inf}}$ pools from which $\SIGMA$ can be recovered w.h.p. Therefore, from an information-theoretic point of view, the probabilistic version of the sublinear quantitative group testing problem is completely understood. Again, the lower bound for QGT immediately transfers to a lower bound for the general pooled data problem. 
Indeed, also the upper bound for $d=2$ immediately transfers to $d \geq 2$ (by the same argument as the one presented in Section \ref{sec:d1} to transfer the algorithmic result). Accordingly, the sublinear pooled data problem can be considered as being fully understood from an information-theoretic perspective.

However, when it comes to the question of \emph{efficient} inference of the label vector, the picture is a rather different one. For this, observe that the results of \cite{feige2020quantitative, gebhard2021quantitative} yield a pooling scheme on $(1+\delta)2 \frac{1 - \theta}{ \theta} k$ pools such that w.h.p., there is exactly one $\sigma \in \{0,1\}^{\{x_1, \ldots, x_n\}}$ that matches the measurement results. Therefore, a naive algorithm to find the ground truth would be to go through all $\binom{n}{k}$ such vectors and see which one explains the measurements. This procedure yields a subexponential time algorithm on the information-theoretically optimal number of pools. To obtain a polynomial-time algorithm, different approaches have been proposed, alas at the expense of a greater number of measurements. First, algorithms from the more general compressed sensing problem, such as basis pursuit, can be used to infer the ground-truth \cite{candes, donoho_l1}. Since the literature on compressed sensing is vast, we refrain from reviewing it in more detail. 
The important point with respect to the pooled data problem is that all compressed sensing algorithms require $\Omega \bc{ k \log n  }$ pools. 
However, these algorithms are designed to also apply to non-integer-valued label vectors. 
Therefore, more specialised algorithms have been proposed especially for the QGT problem. Yet, up to date, also all these tailor-made algorithms rely on $\Omega \bc{ k \log n  }$ pools to ensure w.h.p. inference \cite{coja_spiv, feige2020quantitative, gebhard2021quantitative, karimi_2019}.
\subsection{Our Contribution}
Our main result closes the apparent multiplicative $\ln n$ gap between information-theoretic and algorithmic achievability. Recall $\mdjackov^{\mathrm{inf}}$ from \eqref{minf}. Formally, our result is then given by the following theorem:
\begin{theorem} \label{thm_main_statement}
    For $d=1$ and any $\theta \in (0,1), \delta > 0$ there exist a randomised, non-adaptive pooling scheme $\vec G$ on
   $$ m \leq 
   (1+\delta) \cdot \frac{1+\sqrt{\theta}}{1-\sqrt{\theta}} \cdot \mdjackov^{\mathrm{inf}} $$
   pools and a polynomial time algorithm that, given $\vec G$ and the measurement results $\hat{\SIGMA}_{\vec G}$, outputs $\SIGMA$ w.h.p.
\end{theorem}

The pooling scheme and algorithm from Theorem \ref{thm_main_statement} immediately imply the following:
\begin{corollary} \label{thm_main_corollary}
    For any $d \in \NN, \theta \in (0,1), \delta > 0$ there exist a randomised, non-adaptive pooling scheme $\vec G$ on
   $$ m \leq 
   (1+\delta) \cdot \frac{1+\sqrt{\theta}}{1-\sqrt{\theta}} \cdot \mdjackov^{\mathrm{inf}} $$
   pools and a polynomial time algorithm that, given $\vec G$ and the measurement results $\hat{\SIGMA}_{\vec G}$, outputs $\SIGMA$ w.h.p.
\end{corollary}

\begin{figure}[h!]
    \centering
    \includegraphics[scale=0.5]{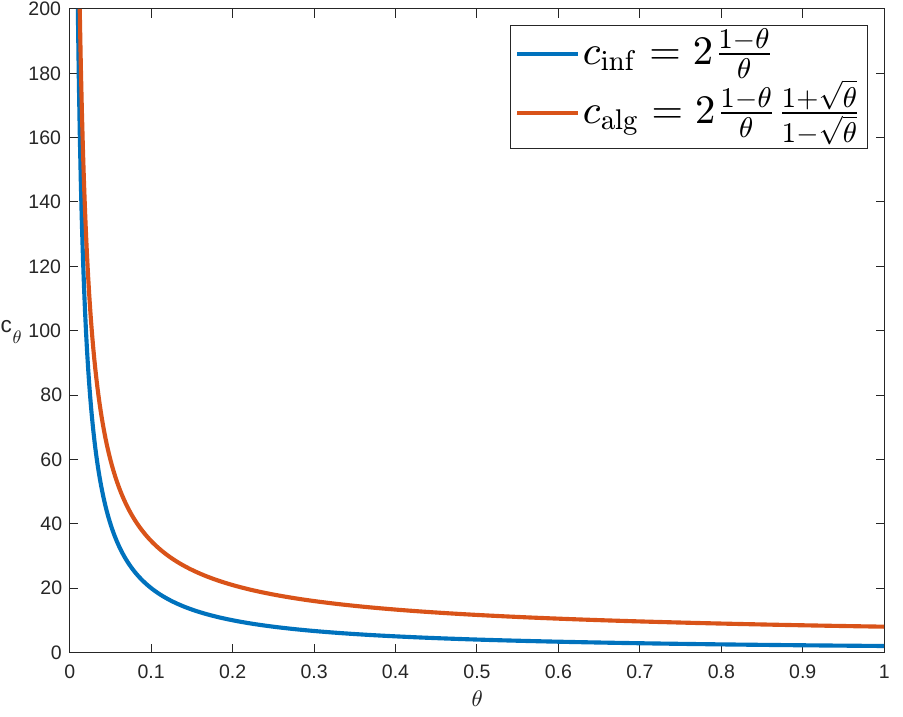}
    \caption{The contrast between the multiplier for the information theoretic bound (blue, $c_{\mathrm{inf}}$) and the current result (red, $c_{\mathrm{alg}}$), when the number of tests is represented by the form $m=c_{\theta}k$.}
    \label{fig:Ctheta}
\end{figure}

Observe that for $\theta \downarrow 0$, the bound in Theorem \ref{thm_main_statement} approaches the information-theoretic threshold, while the additional factor of $(1+\sqrt{\theta})/(1-\sqrt{\theta})$ becomes worse as $\theta \uparrow 1$ (see Figure \ref{fig:Ctheta}). We beat the previous $\ln n$ barrier through the design of a novel pooling scheme, which is based on the spatial coupling technique from coding theory. The latter is augmented by the introduction of ``local'' measurements that allow to determine the number of items of a specific label within strategically chosen ``smaller'' subgroups. This novel design enables a threshold-based, one-stage decoding algorithm to draw conclusions based on this precise local picture. Indeed, the resulting ``local prior'' is necessary to control potentially large deviations from the global prior. As \algoname \hspace{0.01cm} is highly inspired by \emph{Approximate Message Passing}, or more precisely, the first step of a standardised variant, it is not surprising that a good guess on the prior improves the algorithm's performance. Therefore, we believe that the construction itself might be of general interest.
\subsection{Discussion and Outlook}
Our main result Theorem \ref{thm_main_statement} completely removes the previous multiplicative $\log n$-gap between the information-theoretic threshold $\mdjackov^{\mathrm{inf}}$ and all designs for efficient algorithms in the sublinear pooled data problem. It therefore shows that the $\ln n$ factor was artificial, and not a consequence of inherent computational hardness of the problem in a certain parameter regime. Similar information-computation gaps have been observed in a series of prominent examples such as the planted clique problem or sparse principal component analysis, and the understanding of their origins is of great scientific interest. 

In the sublinear \emph{(binary) group testing} problem, such a gap has been recently closed completely \cite{coja_spiv}. The binary group testing problem is similar to the QGT problem, as again the goal is to infer a uniform label vector $\SIGMA \in \{0,1\}^{\{x_1, \ldots, x_n\}}$ that has exactly $k=n^{\theta}$ ones. The important difference to the QGT problem is that each pooled measurement outputs only whether there is at least one item of label $1$ in the corresponding pool. 
In \cite{coja_spiv} closed a constant factor gap of size approximately $\log^{-1}(2)$ by adapting the spatial coupling method from coding theory \cite{aco_spatialcoupling_felstrom, kudekar_2013}. 
While we also employ a modified variant of the spatial coupling design in the present article, key changes in the actual implementation are crucial to make the general approach work.

Put simply, spatial coupling induces an order to the items by dividing them into separate compartments. 
This procedure generally facilitates learning of the labels, as an inference algorithm can decode items compartment by compartment in a way that uses the information about the labels in prior, already decoded compartments. 
In the binary group testing problem, a thresholding algorithm based on the number of pools that an item is part of and which do not contain an item with label $1$ from the previous compartments turned out to be powerful enough to infer the labels w.h.p. 
An analogous approach for the QGT problem would then use a thresholding algorithm that looks at the sum of the pools that an item is part of, minus the labels of all items from previous compartments that appear in these pools. 
Unfortunately, it is straightforward to verify that this approach based on \emph{unexplained neighbourhood sums} does not remove the additional $\log n$-factor.
The main source of failure in this approach lies in reverse trends in the size of partial neighbourhood sums and their information content (see Section \ref{SSec:UnexpNbhd}). 

To overcome this barrier, we introduce an augmentation of the compelling spatial coupling design, which might be of independent interest in other problems as well, such as non-Bayes optimal inference problems, where the global prior is unknown or estimated falsely.
Indeed, our main modification of the basic spatial coupling design is the introduction of an asymptotically negligible number of pools at the right places, which measure the numbers of items of each label within the compartment. This ``local'' information allows \algoname \hspace{0.01 cm} to compute a more accurate score function for each item. 
It also enables an application of \algoname \hspace{0.01 cm} to the setting where each item is assigned a label independently with probability $k_i/n$ for label $i \in \{0, \ldots, d\}$. 
This is a consequence of the observation that i.i.d. sampling of the labels will lead to the same expected numbers of items of each label as in our present setting with negligible fluctuations. The exact numbers of items of each label will be determined during the inference procedure through the local, additional pools, allowing to proceed as in the current article.

An extended abstract of an earlier variant of this work was presented at the Conference on Learning Theory (COLT) 2022 \cite{hahnklimroth_2022_qgt}. The current version improves the underlying pooling scheme and therefore requires a constant factor less measurements than \cite{hahnklimroth_2022_qgt}.

The most conspicuous open problem that arises from the current work is whether the remaining constant multiplicative gap between the information-theoretic threshold and our bound is an artefact of \algoname, its performance analysis, or whether there indeed is a gap between exponential time constructions and efficient algorithms of a certain type, such as algorithms expressed by low-degree polynomials. A similar gap with regard to this rich class of algorithms has recently been discovered in other related inference problems \cite{bandeira2022franz, pmlr-v178-coja-oghlan22a, wein2022optimal}. We conjecture that our bound is not the final answer, since we employ a rather naive union bound to estimate the overall error probability, while in reality, the neighbourhood sums of items within the same compartment should be highly correlated.

Finally, let us briefly elaborate on the condition that each item is included in a given pool at most once, since to the best of our knowledge, this seems to have been assumed without further specification in previous work on the topic. Indeed, if there is no such restriction, then the problem can be solved efficiently by a \emph{single} test. This single pool is constructed by including each item $x_i, i=1, \ldots, n$, $d^{i-1}$ times, yielding a histogram that can be transformed through summation into a test result of 
\begin{align*}
    \hat \SIGMA = \sum_{i=1}^n d^{i-1} \SIGMA_{i}.
\end{align*}
The decoding scheme would then simply convert $\hat \SIGMA$ into its corresponding $d$--ary string of length $n$.
\subsection{Notation}\label{sec_notation}
Throughout the article, for $n \in \NN$, we often abbreviate $[n]:=\{1, \ldots, n\}$ and $[n]_0:=\{0,1, \ldots, n\}$.  If $z \in \RR^p$, then we also use the notation $z(t)$ to refer to the $t$-th component of $z$, where $t \in [p]$. We also adopt the convention that random quantities are indicated by bold letters.

Let 
\begin{align}
    V_{\mathrm{bulk}}:=\{x_1, \ldots, x_n\}
\end{align}
be the set of items. 
We aim to infer the labels $\SIGMA_1, \ldots, \SIGMA_n \in [d]_0$ of the $n$ items in $V_{\mathrm{bulk}}$, where the vector $\SIGMA:= (\SIGMA_1, \ldots, \SIGMA_n) \in [d]_0^n$ is chosen uniformly at random from all vectors containing exactly $k_i$ entries of value $i$ for each $i \in [d]$.
Here, $k_i = \lfloor n^{\theta_i}\rfloor$ for $i \in [d]$, and $\theta_i \in (0,1)$. We let $\theta = \max\{\theta_i: i \in [d]\}$ and abbreviate $k = \max k_i$ such that $k = \lfloor n^{\theta}\rfloor$.

The available information about the labels is given by label-histograms from $m$ subsets $a_1, \ldots, a_m$ of $V_{\mathrm{bulk}}$. We call these subsets pools. 
In the following, we will represent pooling schemes $\{a_1, \ldots, a_m\}$ as bipartite graphs, as described in Section \ref{sec:model}. In this representation, the pools and the items yield the two vertex classes of the bipartite graph (see Figure \ref{fig:example}).
We denote the set of tests that item $x_i$ participates in by $\partial x_i$ and the items partaking in pool $a_j$ by $\partial a_j$. These can analogously be thought of as the neighbourhoods of the associated vertices in the bipartite graph. Finally, we denote the label-histogram of pool $a_j$ by $\hat \SIGMA_j \in [n]^{d+1}$ such that 
$\hat \SIGMA_j(\omega) = \sum_{i: x_i \in \partial a_j} \vecone \cbc{\SIGMA_i = \omega}$. For $d=1$, the case of \emph{quantitative group testing}, we let $\hat \SIGMA_j = \sum_{i: x_i \in \partial a_j} \vecone \cbc{\SIGMA_i = 1} \in \NN$ and $\SIGMA^{-1}(\{1\})=\{x_i \mid \SIGMA_i =1 \}$.

We denote by Hyp$(N,M,K)$ the \textit{hypergeometric distribution} with parameters $N \in \NN_0, M,K \in \cbc{0, \ldots, N}$, which is defined by its point masses
\begin{align*}
    \text{Hyp}(N,M,K)\bc{\cbc{j}} = \frac{\binom{M}{j}\binom{N-M}{K-j}}{\binom{N}{K}}  \text{ for} \quad j \in \cbc{\max\cbc{0, K+M-N}, \ldots, \min\cbc{K,M}}.
\end{align*}
Finally, a list of all relevant notation that is used throughout the article can be found in Appendix \ref{appendix_notation}.
\subsection{Outline} 
The remainder of the manuscript is organised as follows. In Section \ref{Sec:Imp}, we present a basic outline of the novel pooling scheme and the inference algorithm \algoname. Section \ref{Sec:Proof} contains the proof of Theorem \ref{thm_main_statement}. In this proof, we first analyse the algorithm idea for general parameters and then optimise their choice. Correspondingly, Section \ref{SSec:Constraints} presents the main restrictions on our general pooling scheme. The following Section \ref{SSec:PoolingScheme} collects the main properties of the pooling scheme that will be used later. Section \ref{SSec:PropUnexpNbhd} then concentrates on the properties of unexplained neighbourhood sums and discusses the thresholding scheme used by \algoname. The section concludes with Section \ref{SSec:Proof}, which proves the correctness of \algoname \hspace{0.01 cm} and therefore the main result. Finally, Section \ref{sec:d1} gives the proof of Corollary \ref{thm_main_corollary}.

\section{Implementation} \label{Sec:Imp}
The efficient inference algorithm of Theorem \ref{thm_main_statement} is composed of two components: a pooling scheme and an efficient inference algorithm which makes use of the pooling scheme. The pooling scheme is based on a spatially coupled design which is a technique that has been used in compressed sensing and coding theory 
\cite{aco_spatialcoupling_felstrom, Kudekar_2010_2}.
On a high level, the spatially coupled design introduces a localised structure to the bipartite pooling scheme graph which allows for greater control over the concentration of the respective labels.
Our algorithm  will then sequentially classify batches of items using a threshold based upon a certain score for each item.
Thanks to Corollary \ref{thm_main_corollary}, we restrict to the case $d=1$ throughout this section.

\subsection{Augmented Spatial Coupling} \label{SSec:SpCoup}
Recall that a pool is a subset of $V_{\mathrm{bulk}}$, which is represented by a factor node along with its neighbourhood in the associated bipartite graph. The pooling scheme that constitutes the basis for \algoname \hspace{0.01 cm} builds upon the introduction of a spatial order to both items and pools. 
For this, we partition the set of items $V_{\mathrm{bulk}}$ into $\ell = o(k)$ \emph{compartments} $V[s], \ldots, V[s + \ell -1] \subset V_{\mathrm{bulk}}$ of (almost) equal sizes $|V[i]| \in \{\lfloor n/\ell \rfloor, \lceil n/\ell \rceil\}$. 
We denote by $\vec k_i[j]$ the random number of items with label $i$ in compartment $V[j]$. In the case of QGT, we write $\vec k[j] =  \vec k_1[j]$ for brevity. 
Analogously, for an integer $m$ divisible by $\ell + s - 1$, we partition the $m$ pools into $\ell + s - 1$ compartments such that each of the compartments contains exactly $m/(\ell + s - 1) \sim m/\ell$ pools. 
More precisely, we split the tests into compartments denoted by $F[1], \ldots, F[\ell + s - 1]$.
The partition of items and pools will allow us to successively infer the labels of compartments $V[i]$ for $i \in \{ s, s+1, \ldots, s+\ell-1\}$, using the information of previous compartments along the way. The precise sizes of all parameters can be found in Section \ref{Sec:Proof}.

Additionally, we extend this setup by a technical novelty: We introduce one additional pool per compartment, which is used to identify the precise number of items of label $1$ within that compartment. This information crucially enables \algoname \hspace{0.01 cm} to correctly guess the ``local'' probability of having label $1$ within the given compartment.
More formally, for each compartment $V[i]$ $(i = s, \ldots, \ell + s - 1)$, we add one pool $A_i$ that contains all items of $V[i]$. 
The corresponding measurement is denoted by $\hat \SIGMA_{A_i}$. 
Observe that this only adds $\ell = o(k)$ additional pools and thus can be disregarded asymptotically.

To facilitate the initial steps of \algoname, we moreover introduce $s-1$ artificial compartments $V[1]$, $\ldots$, $V[s-1]$. Each of the artificial compartments  contains $\lceil n/\ell \rceil$ \emph{auxiliary items}, such that, in total, we have introduced $n' := (s-1) \lceil n/\ell \rceil$ auxiliary items.
In the remainder of this article, we call the $s-1$ compartments $V_{\text{seed}} = V[1]\cup \ldots \cup V[s-1]$ the \emph{seed}, while the remaining compartments $\ell$ constitute the \emph{bulk} $V_{\text{bulk}}$. The complete set of items will then be $V= V_{\text{seed}} \cup V_{\text{bulk}} $.
For the auxiliary items, we sample an assignment $\SIGMA' \in \cbc{0,1}^{n'}$ uniformly at random from all vectors with exactly $\vec k' = \lceil(s-1)k \ell^{-1}\rceil$ ones. This only takes polynomial time and in particular, $\SIGMA'$ is known. 

Any pool in the pooling scheme will contain exactly $\Gamma$ items. More precisely, pool $a \in F[i]$ chooses exactly $\Gamma/s$ items from each of the previous variable compartments $V[i-(s-1)], \ldots V[i]$\footnote{Here and in the following, for $r=0, \ldots, s-2$, we identify $V[-r]$ with $V[\ell + s - 1 - r]$ as well as $F[\ell + s + r]$ with $F[r+1]$, which equips the random bipartite graph with a ring structure.}. In each of these $s$ compartments, the items are chosen uniformly without replacement and independently for different compartments and pools. The number $s$ is called the \textit{sliding window}. The terminology and construction are depicted in Figure \ref{Fig_spatial_coupling_idea}, which also illustrates the ring structure of the compartments by showing the compartments surrounding the seed.

\begin{figure}[ht!]
\begin{tikzpicture}[scale=0.85]

\foreach \i in {1,...,12}
{
        \def\lab{x_\i};
        \node[circle,draw=black!60!green,fill=green!30,minimum size=1] (\lab) at (0.4*\i,0) {};
}
\foreach \i in {13,...,20}
{
        \def\lab{x_\i};
        \node[circle,draw=black, color=blue,minimum size=1, fill=blue!30] (\lab) at (0.4*\i,0) {};
}
\foreach \i in {21,...,36}
{
        \def\lab{x_\i};
        \node[circle,draw=black!60!green, fill=green!30,minimum size=1] (\lab) at (0.4*\i,0) {};
}
\foreach \i in {0,...,9}
{
        \def\x{4*\i};
        \draw[dashed] (0.4*\x+0.2,0.5) -- (0.4*\x+0.2,-2.5);
}
\foreach \i in {1,...,9}
{
        \def\labx{c_\i};
        \def\labaone{done_\i};
        \def\labatwo{dtwo_\i};
        \pgfmathsetmacro{\xcoord}{0.4*(\i+1.5)+3*0.4*(\i-1))};

        \coordinate (\labx) at (\xcoord,-0.3);
        \coordinate (\labaone) at (\xcoord-0.25,-1.4);
        \coordinate (\labatwo) at (\xcoord+0.25,-1.4);
}

\foreach \j in {2,3,4,5,6,7,8,9,10}{
    \pgfmathsetmacro{\ione}{2*\j-1};
    \pgfmathsetmacro{\itwo}{2*\j};
    \def\laba{a_\ione};
    \def\labb{a_\itwo};
    \pgfmathsetmacro{\xcoordone}{0.4*4*(\j-2)+1.4};
    \pgfmathsetmacro{\xcoordtwo}{0.4*4*(\j-2)+1.0};
    \pgfmathsetmacro{\xcoordthree}{0.4*4*(\j-2)+0.6};
    \node[rectangle, minimum size=8,draw=black!60!green] (\laba) at (\xcoordone,-1.7){};
    \node[rectangle, minimum size=8,draw=black!60!green] (\labb) at (\xcoordtwo,-1.7){};
    \node[rectangle, minimum size=8,draw=black!60!green] (\laba) at (\xcoordthree,-1.7){};

}

\foreach \i in {4,5}
{
    \filldraw[fill=blue!20] (c_\i) -- (done_\i) -- (dtwo_\i) -- cycle;
}

\foreach \i in {4,5}
{
    \pgfmathsetmacro{\x}{\i+1};
    \filldraw[fill=blue!40] (c_\i) -- (done_\x) -- (dtwo_\x) -- cycle;
}
\foreach \i in {4,5}
{
    \pgfmathsetmacro{\x}{\i+2};
    \filldraw[fill=blue!60] (c_\i) -- (done_\x) -- (dtwo_\x) -- cycle;
}

\foreach \i in {1,2,3,6,7,8,9}
{
    \filldraw[fill=black!60!green] (c_\i) -- (done_\i) -- (dtwo_\i) -- cycle;
}

\foreach \i in {1,2,3,6,7,8}
{
    \pgfmathsetmacro{\x}{\i+1};
    \filldraw[fill=black!40!green] (c_\i) -- (done_\x) -- (dtwo_\x) -- cycle;
}
\foreach \i in {1,2,3,6,7}
{
    \pgfmathsetmacro{\x}{\i+2};
    \filldraw[fill=black!20!green] (c_\i) -- (done_\x) -- (dtwo_\x) -- cycle;
}

\filldraw[black!40!green] (0.2,-0.7) -- (1.2,-1.4)--(0.7,-1.4)--(0.2,-0.9)--cycle;
\filldraw[black!20!green] (0.2,-1.1) -- (1.2,-1.4)--(0.7,-1.4)--(0.2,-1.2)--cycle;
\filldraw[black!20!green] (0.2,-0.5) -- (2.8,-1.4)--(2.35,-1.4)--(0.2,-0.55)--cycle;
\filldraw[black!40!green] (13.8,-0.3) -- (14.6,-0.7)--(14.6,-0.9)--cycle;
\filldraw[black!20!green] (13.8,-0.3) -- (14.6,-0.5)--(14.6,-0.55)--cycle;
\filldraw[black!20!green] (12.2,-0.3) -- (14.6,-1.1)--(14.6,-1.2)--cycle;

\filldraw[black!30!red] (1,1.6) -- (0.4, 0.3)--(1.6,0.3)--cycle;
\filldraw[black!30!red] (2.6,1.6) -- (2.0, 0.3)--(3.2,0.3)--cycle;
\filldraw[black!30!red] (4.2,1.6) -- (3.6, 0.3)--(4.8,0.3)--cycle;
\filldraw[black!30!red] (9.0,1.6) -- (8.4, 0.3)--(9.6,0.3)--cycle;
\filldraw[black!30!red] (10.6,1.6) -- (10.0, 0.3)--(11.2,0.3)--cycle;
\filldraw[black!30!red] (12.2,1.6) -- (11.6, 0.3)--(12.8,0.3)--cycle;
\filldraw[black!30!red] (13.8,1.6) -- (13.2, 0.3)--(14.4,0.3)--cycle;

\node (A)[text=black!60!green] at (1,2.5) {$V[7]$};
\node (B)[text=black!60!green] at (2.63,2.5) {$V[8]$};
\node (C)[text=black!60!green] at (4.26,2.5) {$V[9]$};
\node[text=blue!80] (D) at (5.87,2.5) {$V[1]$};
\node[text=blue!80] (E) at (7.45,2.5) {$V[2]$};
\node[text=black!60!green]  (F) at (9.05,2.5) {$V[3]$};
\node (G)[text=black!60!green] at (10.65,2.5) {$V[4]$};
\node (H)[text=black!60!green] at (12.28,2.5) {$V[5]$};
\node (I)[text=black!60!green] at (13.91,2.5) {$V[6]$};

\node (J)[] at (1,-2.3) {{\tiny\textcolor{black!60!green}{$F[7]$}}};
\node (K)[] at (2.63,-2.3) {{\tiny\textcolor{black!60!green}{$F[8]$}}};
\node (L)[] at (4.25,-2.3) {{\tiny\textcolor{black!60!green}{$F[9]$}}};
\node (M)[] at (5.83,-2.3) {{\tiny\textcolor{black!60!green}{$F[1]$}}};
\node (N)[] at (7.4,-2.3) {{\tiny\textcolor{black!60!green}{$F[2]$}}};
\node (O)[] at (8.97,-2.3) {{\tiny\textcolor{black!60!green}{$F[3]$}}};
\node (P)[] at (10.6,-2.3) {{\tiny\textcolor{black!60!green}{$F[4]$}}};
\node (Q)[] at (12.22,-2.3) {{\tiny\textcolor{black!60!green}{$F[5]$}}};
\node (R)[] at (13.84,-2.3) {{\tiny\textcolor{black!60!green}{$F[6]$}}};
\node at (-0.2,-1){$\cdots$};
\node at (14.9,-1){$\cdots$};

\end{tikzpicture}
\caption{Schematic representation of the pooling scheme with $n = 28$ items, $\ell = 9$ compartments, $m=18$ measurements and a sliding window of size $s = 3$. The number of (blue) auxiliary items, whose weight is known a priori, is $8$ in this example. The additional red pools on top of the graph test all items in the corresponding compartments to identify the local prior. In the seed, the local prior is known anyways. All other parameters are left unspecified in the figure.
}
\label{Fig_spatial_coupling_idea}

\end{figure}
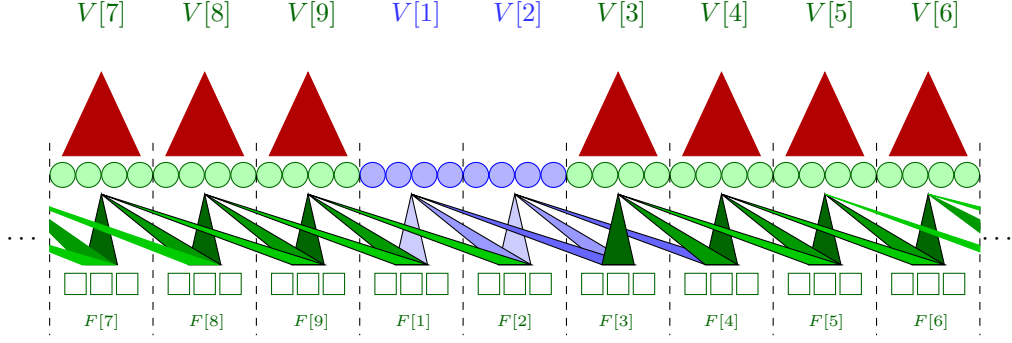

As already mentioned, any pooling scheme can be seen as a bipartite graph. We call the random bipartite graph corresponding to our pooling scheme $\vec G$.  
In this graph, the degrees of the pools are fixed to $\Gamma = \abs{ \partial a_j }$ for all $j \in [m]$. We denote the (random) degrees of the items $x_1, \ldots, x_n$ into the regular pools by $\vec \Delta_{x_1}, \ldots, \vec \Delta_{x_n}$, that is, $\vec \Delta_{x_i} = \abs{ \cbc{ a \in \partial x_i \cap \bc{F[1] \cup \ldots \cup F[\ell]}  } }$.  The number of neighbours of item $x_i$ among the elements of $F[i+j]$ is denoted by $\vec \Delta_{x_i}[j]$, where $j=0, \ldots, s-1$. Finally, we set $\Delta := \Erw[\vec \Delta_{x_1}]$. The notation used in the spatial coupling scheme is summarised in Appendix \ref{appendix_notation}.
\subsection{Unexplained Neighbourhood} \label{SSec:UnexpNbhd}
The previous subsection has set up the pooling scheme. To define the precise threshold procedure that our inference algorithm \algoname \hspace{0.01 cm} is based on, we define a few more quantities. 
In the following, we denote the current estimate of $\SIGMA$ of \algoname \hspace{0.01 cm}  by $\tilde \sigma$. Initially, $\tilde \sigma$ is set to $\SIGMA'$ on the coordinates corresponding to $V_{\text{seed}}$ and to zero on the coordinates corresponding to $V_{\text{seed}}$.

Let $x \in V[i]$ for $i \in \{s, \ldots, \ell + s -1\}$ be an item of the bulk and $j \in \{0, \ldots, s-1\}$.
We then set 
\begin{align}\label{def_unexplained_alg} 
\cU^j_x := \cU^j_x(\tilde \sigma, \vec G) = \sum_{a \in \partial x \cap F[i + j]} \bc{\hat \SIGMA_a - \sum_{r=i-s+j+1}^{i-1} \sum_{y \in \partial a \cap V[r]} \tilde \sigma_y}
\end{align}
and
\begin{align}\label{def_unexplained} 
\vec U^j_x := \vec U^j_x(\SIGMA, \vec G) = \sum_{a \in \partial x \cap F[i + j]} \bc{\hat \SIGMA_a - \sum_{r=i-s+j+1}^{i-1} \sum_{y \in \partial a \cap V[r]} \SIGMA_y}.
\end{align}
Intuitively speaking, $\vU^j_x$ counts the number of items of label $1$ in tests adjacent to $x$ that are part of the same or \emph{later} compartments than $x$ with respect to the actually unknown ground-truth $\SIGMA$. In contrast, $\cU_x^j$ can be seen as an estimate of $\vec U^j_x$ that is based on $\tilde \sigma$ rather than the true $\SIGMA$. Indeed, if all labels in compartments $V[1], \ldots, V[i-1]$ have been decoded correctly, $\vU^j_x = \cU_x^j$. We call $\vU^j_x$  (or $\cU_x^j$) the \emph{unexplained neighbourhood sum} of item $x$ with respect to the tests in compartment $F[i+j]$. Compartmentwise consideration of neighbourhood sums is important here, as every item $x$ is a member of tests from different compartments, and during inference, tests from each compartment will have varying proportion of already inferred neighbours.

Next, let 
\begin{align}\label{def_sigma_Ex}
  \cE_x := \mathcal F\bc{\partial x, \SIGMA_x, (\vec k[r])_{r \in [\ell+s-1]}}  
\end{align}
be the $\sigma$-algebra generated by the tests that $x$ is part of, its label as well as the numbers of items with label $1$ in all compartments. 
We define an estimate of $\Erw\brk{\vec U_x^j \vert \cE_x} -  \vec \Delta_x[j]\SIGMA_x$, given $\partial x, (\vec k[r])_{r \in [\ell+s-1]}$, as
\begin{align}\label{def_cent}
    \mathcal M_x^j :=  
    \vec\Delta_x[j] \sum_{r=0}^{j} \bc{\frac{\Gamma}{s}-\vecone\cbc{r=0}} \frac{\vec k[i+r]}{|V[i+r]\setminus \{x\}|}.
    \end{align}
Furthermore, define
\begin{align}\label{nbh_sums_alg}
    \cN_x^j := \frac{\cU_x^j - \mathcal M_x^j }{(j+1)\frac{\Delta}{s}\sqrt{\frac{k}{m}}} \quad \text{and} \quad \cN_x := \sum_{j=0}^{s-1} \cN_x^j,
\end{align}
and accordingly
\begin{align}\label{nbh_sums}
\vec N_x^j := \frac{\vec U_x^j - \Erw\brk{\vec U_x^j \vert \cE_x} + \vec \Delta_x[j]\SIGMA_x}{(j+1)\frac{\Delta}{s}\sqrt{\frac{k}{m}}} \quad \text{and} \quad \vec N_x := \sum_{j=0}^{s-1} \vec N_x^j.
\end{align}
Clearly, $\vec N_x^j$ is with respect to knowledge of the actual ground truth and is a standardised variant of $\vU^j_x$ which takes the influence of the label of item $x$ itself into account: if item $x$ has label $1$, it increases the unexplained neighbourhood sum deterministically by the number of adjacent tests in the compartment. The standardisation is necessary because the size of $\vec U_x^j$ dramatically increases for increasing $j$ such that the unexplained sums into different compartments would not be comparable. The normalised neighbourhood sum $\vec N_x$ finally is a weighted sum that takes care of the effect that the information in compartments close to $x$ are much more valuable: For $x \in V[i]$, only the items from $V[i]$ are unknown in compartment $F[i]$, whereas in $F[i+s-1]$ all items are unknown. Again, $\cN_x^j$ and $\cN_x$ are the corresponding quantities that are based on the estimate $\tilde \sigma$ instead of $\SIGMA$.

\subsection{The Spatially Coupled Inference through Efficient Neighbourhood Thresholding Algorithm} \label{SSec:InfAlg}
Our inference algorithm, called \algoname, now proceeds in a single stage. It will infer the labels of the items compartment by compartment, starting with compartment $V[s]$. 
Suppose therefore that the algorithm is about to infer the labels within compartment $V[i]$ for $i \in \{s, \ldots, s+\ell-1\}$, and that the current estimate of the true labels $\SIGMA$ is $\tilde \sigma$.

\algoname\hspace{0.01cm} then computes the score $\cN_x$ for all $x \in V[i]$ in parallel, using the current estimate $\tilde \sigma$, which encodes information from previously inferred compartments. Since both the pooling scheme as well as the number of label 1 items are known at this point thanks to the additional pools $A_s, \ldots, A_{s+\ell-1}$, all the information needed to compute $\cN_x$ is accessible to the algorithm.
\algoname\hspace{0.01cm}  will then declare any item with a score above a predetermined threshold $T_{\alpha}$ as having label $1$ and the others as having label $0$. This will lead to an update of $\tilde \sigma$ on the coordinates corresponding to compartment $V[i]$. Proceed now, step by step, with the following compartments until all items are labelled. An illustration and formal statement of \algoname\hspace{0.01cm} are given in Figure \ref{Fig:Thresh} and Algorithm \ref{Algo_statement}, respectively.
\vspace{-10pt}
\usetikzlibrary{decorations.pathreplacing}
\tikzset{mybrace/.style={decoration={brace,raise=1.8mm},decorate}}
\begin{figure}[h!]

\centering
\begin{tikzpicture}
\foreach \vertpos in {0}{
    \draw (0,\vertpos) -- (8,\vertpos);
    \foreach \pos/\descr in {0/0,4/{T_\alpha},8/{\Erw\brk{\vec N_x \vert \SIGMA_x=1}}}
        {
        \draw (\pos,\vertpos) -- ++(0,1mm);
        \draw (\pos,\vertpos) -- ++(0,-1mm);
        \node[yshift=-4mm] at (\pos,\vertpos) {$\descr$};
        }
}
\foreach \mycoord in {(0,1)}
     \draw [<-|,blue] \mycoord -- node[above, yshift=2mm,blue]{Set $\tilde \sigma_x = 0$ } ++(4,0); 

\foreach \mycoord in {(4,-1)}
     \draw [|->,red] \mycoord -- node[below, yshift=-2mm,red]{Set $\tilde \sigma_x = 1$ } ++(4,0); 

\end{tikzpicture}
\caption{Illustration of the threshold $T_\alpha$ for the \algoname\hspace{0.01cm} algorithm.}
\label{Fig:Thresh}
\end{figure}

\vspace{-30pt}
\begin{algorithm}
Set $\tilde \sigma = (\SIGMA',0) \in \cbc{0,1}^{n'+n}$; \\
\For{$i = s+1, \ldots, \ell+s-1$}{ 
    \For{any item $x \in V[i]$ calculate $\cN_x$}{
    \If{$\cN_x>T_{\alpha}$}{Set $\tilde \sigma_x = 1$;}}
}

Return $\tilde{\sigma}$.\\
\caption{\algoname\hspace{0.01cm} algorithm.}
\label{Algo_statement}
\end{algorithm}
\vspace{-20pt}
\section{Proof of Theorem 1} \label{Sec:Proof}
In this section, we prove Theorem \ref{thm_main_statement}. For now, all relevant quantities will be expressed in terms of parameters that we will fix only at the end of the proof. A succinct directory of the parameters and their definitions is contained in Appendix \ref{appendix_notation}.

For any $\delta>0$, the number $m$ of conducted measurements will be parameterised as
\begin{align} \label{mdelta}
    m:= 2  c_{\theta} (1+\delta) \bcfr{1-\theta}{\theta} k = (1+\delta) \cdot c_{\theta} \cdot \mdjackov^{\mathrm{inf}}. 
\end{align}
Thus, $c_{\theta}$ denotes the constant, but $\theta$-dependent multiplicative parameter that separates our design from the information-theoretic bound. 
For $\eps_{1}, \eps_2, \eps_{3} >0$ to be specified later (and that depend on $\delta, \theta$), let
\begin{align}
    \ell=  k^{1-\eps_2} = n^{\theta(1-\eps_2)}, \qquad s = k^{1-\eps_3} = n^{\theta(1-\eps_3)} \label{Eq:LS}
\end{align}
as well as
\[ \abovedisplayskip=10pt \Gamma = \frac{n^{1-\eps_{1}}}{2} +O(s) \]
be an integer divisible by $s$. 
This parametrisation of $\Gamma$ follows from a heuristic calculation which shows that in order to obtain a $c_{\theta}$ close to $1$, $\Gamma$ should be chosen close to $n/2$, if each test chooses $\Gamma$ items uniformly.

\subsection{Constraints on Parameters}\label{SSec:Constraints}
We next derive necessary constraints on the parameters for our general approach to work. 
Clearly, the sliding window $s$  should be asymptotically smaller than the number of compartments $\ell$. To ensure that $s = o(\ell)$ in the parametrisation \eqref{Eq:LS}, we assume throughout that
\begin{align}
    \eps_2<\eps_3. \label{Eq:Cond1}
\end{align}
Next, restricting the choice of the $\Gamma$ items for each pool to $s$ compartments of size $n/\ell$, where each item can only be chosen once for each test, yields the condition $\Gamma \leq sn/\ell$.

In terms of $\eps_1, \eps_2$ and $\eps_3$, this reduces to requiring that
\begin{align}
\theta(\eps_2-\eps_3)+\eps_1 \geq 0. \label{Eq:CondA}
\end{align}
Finally, for the neighbourhood sums to exhibit the necessary concentration properties, we additionally require that the average number of neighbours of each item within a specific pool compartment is growing: $\Delta/s = \omega(1)$. Hence, we desire that
\begin{align}
    \eps_3 > \frac{\eps_1}{\theta}. \label{Eq:CondB}
\end{align}

In the next subsection, we prove basic properties of the pooling scheme. For the reader's convenience, we fix a parametrisation of the hypergeometric distribution and list some well-known concentration inequalities as references in Appendix \ref{Appendix_inequalities}. Section \ref{SSec:PoolingScheme} collects some properties of the typical degrees in our spatial coupling scheme which follow immediately from those concentration inequalities.

\subsection{Properties of the Pooling Scheme}\label{SSec:PoolingScheme}
This section provides bounds on the degrees of the items in the pooling scheme as well as on the numbers of defective items per compartment. For this, let $i \in [\ell + s - 1]$ be a compartment index and $j \in [s-1]_0$. Recall that for $x \in V[i]$, $\vec \Delta_x[j]$ denotes the number of pools $a \in F[i+j]$ that $x$ is part of. Furthermore, we set
\begin{align*}
\vec \Delta_x = \sum_{j=0}^{s-1} \vec \Delta_x[j] \qquad \text{as well as} \qquad \Delta = \Erw \brk{\vec \Delta_x}.
\end{align*}
Throughout this section, we assume the parametrisation as specified in Section \ref{SSec:SpCoup}. Moreover, we will neglect errors that arise from the requirement that all involved quantities are integers, as these are asymptotically negligible. The following claim identifies the expected compartment-wise and overall item degrees:

\begin{claim}\label{claim_delta}
    We have
    \begin{align}
        \Delta = \frac{m}{2} n^{ - \eps_1} + o\bc{n^{\theta - \eps_1}} \quad \text{and} \quad \frac{\Delta}{s} = \frac{m}{2s} n^{ - \eps_1} + o\bc{n^{\theta \eps_3-\eps_1}}.
    \end{align}
\end{claim}

\begin{proof}
Recall that $\Gamma$ is an integer divisible by $s$. Then
    \begin{align*}
        \Delta = \sum_{j=0}^{s-1} \Erw\brk{\vec \Delta_x[j]} = s \cdot \frac{m}{\ell+s-1} \cdot \frac{\binom{|V[i]|-1}{\Gamma/s-1}}{\binom{|V[i]|}{\Gamma/s}} = s \cdot \frac{m}{\ell+s-1} \cdot \frac{\Gamma}{s |V[i]|} = \frac{\Gamma m}{(\ell+s-1)|V[i]|}.
        \end{align*}
Plugging in the parameters from Section \ref{SSec:SpCoup} yields the claim.
\end{proof}

\begin{claim}\label{cor_conc_delta}
With probability $1 - o(n^{-2})$, 
\begin{align*}
    \frac{\Delta}{s} -  \sqrt{ \frac{\Delta}{s} } \log{n} \leq \min_{\substack{x \in V_{\mathrm{bulk}}, \hspace{0.05 cm} j \in [s-1]_0}} \vec \Delta_x[j] \leq \max_{x \in V_{\mathrm{bulk}}, \hspace{0.05 cm} j \in [s-1]_0} \vec \Delta_x[j] \leq  \frac{\Delta}{s} + \sqrt{ \frac{\Delta}{s}}\log{n}.
\end{align*}
\end{claim}
\begin{proof}
Each pool $a \in F[i+j]$ chooses $\Gamma/s$ items in compartment $V[i]$ uniformly at random and independently of the other pools. Thus,
\[\vec\Delta_x[j] \sim \Bin\bc{\frac{m}{\ell+s-1},\frac{\Gamma}{s|V[i]|}}.\]
Moreover, $\Erw\brk{\vec\Delta_x[j]} = \Delta/s$ and by the Chernoff bound (\Lem~\ref{lem_chernoff}),
\begin{align*} 
    \Pr\bc{ \abs{\vec \Delta_x[j] - \frac{\Delta}{s}} \geq  \sqrt{\frac{\Delta}{s}}\ln n} \leq \exp\bc{-\Theta\bc{\ln^2 n}}.
\end{align*}
The claim now follows from a union bound over all items $x$ and neighbouring compartments $j$.
\end{proof}

\begin{claim}\label{cor_conc_k}
With probability $1 - o(n^{-2})$, 
\begin{align*}
    \frac{k}{\ell} -   \sqrt{ \frac{k}{\ell} } \log{n} \leq \min_{i \in [\ell+s-1]} \vec k[i] \leq \max_{i \in [\ell+s-1]} \vec k[i] \leq  \frac{k}{\ell} + \sqrt{ \frac{k}{\ell} }\log{n}.
\end{align*}
\end{claim}
\begin{proof} 
The number $\vec k[i]$ of defective items in compartment $V[i]$ has distribution
\[\vec k[i] \sim \Hyp\bc{n,k,|V[i]|}.\]
Moreover, $\Erw\brk{\vec k[i]} = |V[i]|\cdot (k/n) = k/\ell + \Theta(n^{\theta-1})$ and by the Chernoff bound for the hypergeometric distribution (\Lem~\ref{Hyp}),
\begin{align*} 
    \Pr\bc{ \abs{\vec k[i] - \frac{k}{\ell}} \geq  \sqrt{\frac{k}{\ell}}\ln n} \leq \exp\bc{-\Theta\bc{\ln^2 n}}.
\end{align*}
The claim now follows from a union bound over all compartments $i \in [\ell+s-1]$.
\end{proof}

\subsection{Properties of Unexplained Neighbourhood Sums} \label{SSec:PropUnexpNbhd}
In this section, we analyse the idealised unexplained neighbourhood sum $\vec U_x$, conditioned on the two possible states of each item. This will enable us to define a sensible threshold for \algoname. Finally, we will derive concentration of $\vec N_x$, upon which the success of the algorithm relies.

\subsubsection{Distribution of Conditional Unexplained Neighbourhood Sums}
We first take a closer look at the distribution of the unexplained neighbourhood sum $\vec U_x^j$, as defined in \eqref{def_unexplained}. For this, recall from \eqref{def_sigma_Ex} that $\cE_x$ is the $\sigma$-algebra generated by the tests that $x$ is part of, its label and the numbers of items with label one in all compartments.

Let $i \in \{s, \ldots, \ell+s-1\}$ be the index of a non-artificial item compartment and $x \in V[i]$.
Conditioned on the information $\cE_x$, for each $j \in \{0, \ldots, s-1\}$, the remaining randomness in the neighbourhood sum $\vec U_x^{j}$ of $x$ is due to the numbers of defective items in the neighbouring tests of $x$ within compartment $F[i+j]$ that have not yet been explained. We model these numbers by a collection of independent hypergeometrically distributed random variables 
\begin{align}\label{def_hypgeo}
   \cbc{ \vH_{x,a}^{(r:j)}: j \in \{0, \ldots, s-1\}, r \in \{0, \ldots, j\}, a \in F[i+j]}, 
\end{align}
which are independent of $\vec\sigma_x$ and $(\vec k[i])_{i \in [\ell+s-1]}$. Here, $\vH_{x,a}^{(r:j)}$ represents the number of defective items that test $a \in \partial x \cap F[i+j]$ draws from item compartment $V[i+r]$, $r=0, \ldots, j$, excluding $x \in V[i]$. Thus,
\begin{align*}
    \vH_{x,a}^{(r:j)} \sim 
    \begin{cases}
     \Hyp \bc{|V[i]|-1, \vec k[i]-\SIGMA_x, \frac{\Gamma}{s}-1}, & r=0,\\
     \Hyp\bc{|V[i+r]|, \vec k[i+r], \frac{\Gamma}{s}}, & r \in [j].
    \end{cases}
\end{align*}
Here, as previously specified, the first parameter denotes the total size of the population, the second one the number of success states in the population, and the third one the number of draws. 
Finally, to render the subsequent expressions for the mean and variance of the (normalised) neighbourhood sums more readable, we will abbreviate the parameters of the hypergeometric distributions:
\begin{align}\label{def_par_hypgeo}
   \bc{N_x^{(r:j)}, \vec K_x^{(r:j)}, n_x^{(r:j)}} := 
    \begin{cases}
     \bc{|V[i]|-1, \vec k[i]-\SIGMA_x,\frac{\Gamma}{s}-1}, & r=0,\\
    \bc{|V[i+r]|, \vec k[i+r], \frac{\Gamma}{s}}, & r \in [j].
    \end{cases}
\end{align}

In the next lemma, we isolate the contribution of $x$ to its unexplained neighbourhood sum:

\begin{lemma}[Conditional distribution of $\vec U_x^j$]\label{cond_dist_U} 

Let $i \in \{s, \ldots, \ell+s-1\}$ be the index of an item compartment, $j \in \{0, \ldots, s-1\}$ a horizon within the sliding window, as well as $x \in V[i]$. 
Then with  $\{ \vH_{x,a}^{(r:j)}\}$ as defined in (\ref{def_hypgeo}),
\begin{align}\label{cond_distr_U}
    \vec U_x^j \stackrel{d}{=} \vec \Delta_x[j] \SIGMA_x +  \sum_{a \in \partial x \cap F[i+j]}\sum_{r=0}^{j} \vec H_{x,a}^{(r:j)} \qquad \qquad  \text{given } \cE_x.
\end{align}
\end{lemma}

\begin{proof}
Recall that $\vec U_x^j$ is the unexplained neighbourhood sum of item $x$ in compartment $i+j$ under perfect knowledge of the true labels in the first $i-1$ compartments. In this sense, \eqref{def_unexplained} can equivalently be written as 
\begin{align}\label{eq_Uxj}
\vec U_x^j = \sum_{a \in \partial x \cap F[i + j]} \sum_{r=0}^{j} \sum_{y \in \partial a \cap V[i+r]} \SIGMA_y.
    \end{align}
Next, we isolate the contribution of $x$  to $\vec U_x^j$. Since this is of size $\vec \Delta_x[j] \SIGMA_x$, we obtain
\begin{align}
\vec U_x^j = \vec \Delta_x[j] \SIGMA_x + \sum_{a \in \partial x \cap F[i + j]} \sum_{r=0}^{j} \sum_{y \in \partial a \cap (V[i+r]\setminus \{x\})} \SIGMA_y.
    \end{align}
Looking at the right hand side, it thus remains to argue that given $\cE_x$, the inner sum is indeed distributed as a sum of independent hypergeometrically distributed random variables with the specified parameters. For this, recall that each test $a \in \partial x \cap F[i + j]$ chooses its neighbourhood within compartment $V[i+r]$ without replacement, independently from the other tests and its choices in other compartments. In particular, given $\cE_x$, for each $a \in \partial x \cap F[i+j]$, the remaining neighbours $y \in \partial a \cap (V[i]\setminus \{x\})$ from compartment $V[i]$ 
are also drawn uniformly without replacement. 
Moreover, given $\cE_x$, also the numbers $\vec k[s], \ldots, \vec k[\ell + s - 1]$ of defective items in each compartment are known, and thus for each $a \in \partial x \cap F[i+j]$ and $r \in \{0, \ldots, j\}$, the contribution $\sum_{y \in \partial a \cap (V[i+r]\setminus \{x\})} \SIGMA_y$ to $\vec U_x^j$ is hypergeometrically distributed. The parameters of the distribution are the total number of items in compartment $V[i+r]$ excluding $x$, the number of defective items in compartment $V[i+r]$ excluding $x$ and the number of items that each pool $a \in F[i+j]$ draws from item compartment $V[i+r]$ excluding $x$.
Plugging in the numbers yields $(N_x^{(r:j)}, \vec K_x^{(r:j)}, n_x^{(r:j)})$ from (\ref{def_par_hypgeo}).
\end{proof}

\begin{corollary}[Conditional mean and variance of $\vec U_x^j$] \label{cor_cond_exp_pd}
Let $i \in \{s, \ldots, \ell+s-1\}$ be the index of an item compartment, $j \in \{0, \ldots, s-1\}$ a horizon within the sliding window, as well as $x \in V[i]$. 
 Then, with $(N_x^{(r:j)}, \vec K_x^{(r:j)},n_x^{(r:j)})$ as in (\ref{def_par_hypgeo}),
\begin{align} \label{eq_expectation_ux} 
\Erw \brk{ \vec U_x^j \Big \vert \cE_x} & = \vec\Delta_x[j]\SIGMA_x +\vec\Delta_x[j] \sum_{r=0}^{j} n_x^{(r:j)} \frac{\vec K_x^{(r:j)}}{N_x^{(r:j)}}
\end{align}
and
\begin{align*}
    \Var \bc{\vec U_x^j \Big \vert \cE_x}
    & = \vec \Delta_x[j]\sum_{r=0}^{j}  n_x^{(r:j)} \frac{\vec K_x^{(r:j)}}{N_x^{(r:j)}}\bc{1 -\frac{\vec K_x^{(r:j)}}{N_x^{(r:j)}}}\bc{1 - \frac{n_x^{(r:j)}-1}{N_x^{(r:j)}-1}}.
    \end{align*}
\end{corollary}

\begin{proof}
This is an immediate consequence of the characterisation  of the conditional distribution of $\vec U_x^j$ from Lemma \ref{cond_dist_U}, and the independence of hypergeometrics in (\ref{def_hypgeo}).
    \end{proof}

We finally use Corollary \ref{cor_cond_exp_pd} to derive the expectation and variance of the centered and normalised neighbourhood sums, which will form the score of an item under perfect knowledge of all previous compartments. For this, recall $\vec N_x^{j}$ as defined in \eqref{nbh_sums}.
\begin{corollary}[Conditional mean and variance of $\vec N_x^j$]\label{cor_parameters_N}
Let $i \in \{s, \ldots, \ell+s-1\}$ be the index of an item compartment, $j \in \{0, \ldots, s-1\}$ a horizon within the sliding window, as well as $x \in V[i]$. 
 Then
\begin{align*}
  \Erw \brk{\vec N_x^j \Big \vert \cE_x} = \frac{\vec\Delta_x[j]\SIGMA_x}{\frac{\Delta}{s}(j+1)\sqrt{\frac{k}{m}}} \quad \text{and} \quad \Var \bc{\vec N_x^j \Big \vert \cE_x}  = \frac{\Var \bc{\vec U_x^j \Big \vert \cE_x}}{\bc{\frac{\Delta}{s}}^2(j+1)^2\frac{k}{m}} .
\end{align*}
In particular, since $\Erw[ \vec \Delta_x[j] \vert \SIGMA_x] \cdot \SIGMA_x = \Delta/s \cdot \SIGMA_x$,
\begin{align*}
    \Erw\brk{\vec N_x^j \vert \SIGMA_x=0} = 0 \qquad \text{and} \qquad \Erw\brk{\vec N_x^j \vert \SIGMA_x=1} = \frac{1}{(j+1)\sqrt{\frac{k}{m}}}.
\end{align*}
\end{corollary}

\subsubsection{Choosing a Threshold}
Going from compartment-wise considerations to the overall score $\vec N_x$ of $x$, Corollary \ref{cor_parameters_N} yields that 
\begin{align*}
    \Erw\brk{\vec N_x \vert \SIGMA_x=0} = 0 \qquad \text{and} \qquad \Erw\brk{\vec N_x \vert \SIGMA_x=1} = \sqrt{\frac{m}{k}}\sum_{j=0}^{s-1} \frac{1}{j+1}.
\end{align*}
Correspondingly, a healthy item should have a score $\vec N_x$ around $0$, while a defective item should have a score of size about $\sqrt{\frac{m}{k}} \ln s$,  taking into account the asymptotics of the harmonic sum. 
Let $\alpha \in (0,1)$. To reliably discern between healthy and defective individuals, we define the parametrised threshold
\begin{align} \label{threshold ansatz}
    T_{\alpha} = \alpha \cdot \sqrt{\frac{m}{k}} \ln s,
    \end{align}
which is asymptotically in between $\Erw\brk{\vec N_x \vert \SIGMA_x=0}$ and $\Erw\brk{\vec N_x \vert \SIGMA_x=1}$. Our approach in the next subsection will be to classify each item with a score below $T_\alpha$ as healthy and each item with score above $T_\alpha$ as defective, and to evaluate the probability of errors under this classification.

\subsubsection{Concentration of $\mathbf{N}_x$} 

The success of the thresholding algorithm chiefly depends on concentration properties of the normalised neighbourhood sums. These are the content of the following lemma and will ensure small enough error probabilities to beat the union bounds, which we cannot do without in the proof of our main theorem. 

\begin{lemma}\label{Lem_conc}
Let $T_\alpha$ be defined as in \eqref{threshold ansatz} and $x \in V[i]$ for $i \in \{s, \ldots, \ell+s-1\}$. 
Then
\small
\begin{align*}
    \Pr\bc{\vec N_x \geq T_{\alpha} \big \vert \SIGMA_x=0} & \leq s^{-(1+o(1)) \frac{\alpha^2}{2}\cdot \frac{m}{k}} + o(n^{-2})\\
    \Pr\bc{\vec N_x \leq T_\alpha \big \vert \SIGMA_x=1}  & \leq  s^{-(1+o(1))\frac{(1-\alpha)^2}{2}\cdot \frac{m}{k}}+ o\bc{n^{-2}}.
    \end{align*}
\normalsize
\end{lemma}

\begin{proof}[Proof of Lemma \ref{Lem_conc}]
Recall the random variables $\{\vec H_{x,a}^{(r:j)}\}$ defined in \eqref{def_hypgeo}. Lemma \ref{cond_dist_U} implies that conditionally on $\cE_x$, 
  \begin{align}\label{eq_conc_1}
      \vec N_x \stackrel{d}{=} \sum_{j=0}^{s-1} \frac{\vec \Delta_x[j] \SIGMA_x}{\frac{\Delta}{s}(j+1)\sqrt{\frac{k}{m}}} + \sum_{j=0}^{s-1} \sum_{a \in \partial x \cap F[i+j]} \sum_{r=0}^{j} \frac{\vec H_{x,a}^{(r:j)} - \Erw[\vec H_{x,a}^{(r:j)} \vert \cE_x]}{\frac{\Delta}{s}(j+1)\sqrt{\frac{k}{m}}}.
  \end{align}
The distributional identity (\ref{eq_conc_1}) shows that given $\cE_x$, $\vec N_x$ is distributed as a constant plus a sum of independent, centered and rescaled hypergeometric random variables. Correspondingly, we will first adopt a point of view on the sum of the hypergeometric random variables that allows to apply a classical concentration inequality to them.
In the remainder of this proof, we abbreviate
\begin{align}\label{eq_mux}
      \vec \mu_x = \sum_{j=0}^{s-1} \frac{\vec \Delta_x[j]}{\frac{\Delta}{s}(j+1)\sqrt{\frac{k}{m}}}.
  \end{align}

For this, we represent each hypergeometric random variable $\vec H_{x,a}^{(r:j)}$ as the sum over $n_x^{(r:j)}$ negatively associated indicator random variables:
Given $\cE_x$, for any $j \in \{0, \ldots, s-1\}, a \in \partial x \cap F[i+j], r \in \{0, \ldots, i\}$, let
\begin{align}
    \vec X_{x,a,1}^{(r:j)}, \ldots, \vec X_{x,a, n_x^{(r:j)}}^{(r:j)}
    \end{align}
be the indicators random variables of including a defective item into pool $a$, when imagining the $n_x^{(r:j)}$ items for pool $a$ from item compartment $V[i+r]$ (excluding $x$ if necessary) being drawn one by one and without replacement. Since the pools draw their $\Gamma/s$ items independently from the different compartments and also independently of each other, these indicators are independent for different choices of $a$ and $r$. However, fixing $a$ and $r$, the $n_x^{(r:j)}$ indicator random variables describing the defective items in pool $a$ from compartment $V[i+r]$  are negatively associated (compare Definition \ref{def_NA}). Thus,
\begin{align*}
    \vec H_{x,a}^{(r:j)}\stackrel{d}{=} \sum_{j=0}^{s-1} \sum_{a \in \partial x \cap F[i+j]} \sum_{r=0}^{j} \sum_{t=1}^{n_x^{(r:j)}} \vec X_{x,a,t}^{(r:j)}
\end{align*}
can be regarded as a single sum over $\sum_{j=0}^{s-1}\vec \Delta_x[j] ((j+1)\Gamma/s - 1)$ negatively associated indicator random variables. We next return to our scaling. Setting
\begin{align*}
     \vec p_x^{(r:j)} = \frac{\vec K_x^{(r:j)}}{N_x^{(r:j)}},
\end{align*}
we have now arrived at the distributional identity
\begin{align}\label{eq_conc_2}
     \vec N_x  \stackrel{d}{=} \vec \mu_x \SIGMA_x + \sum_{j=0}^{s-1} \sum_{a \in \partial x \cap F[i+j]} \sum_{r=0}^{j} \sum_{t=1}^{n_x^{(r:j)}} \frac{ \vec X_{x,a,t}^{(r:j)} - \vec p_x^{(r:j)}}{\frac{\Delta}{s}(j+1)\sqrt{\frac{k}{m}}}.
\end{align}
In particular, (\ref{eq_conc_2}) shows that $\vec N_x \vert \cE_x$ is distributed as a constant term plus a sum of 
negatively associated, bounded and centered random variables, each of which has a very small variance 
thanks to our scaling. Because of the small variance of the individual summands, we will apply Bernstein's inequality for negatively associated random variables (Theorem \ref{Bernstein_Negative}) to show concentration of the sum. For this, let
\begin{align*}
    \vec v_x^2 :=  \sum_{j=0}^{s-1} \sum_{a \in \partial x \cap F[i+j]} \sum_{r=0}^{j} \sum_{t=1}^{n_x^{(r:j)}}\frac{  \Var(\vec X_{x,a,t}^{(r:j)}\vert \cE_x) }{(j+1)^2 \bc{\frac{\Delta}{s}}^2\frac{k}{m}}.
    \end{align*}
Then by Theorem \ref{Bernstein_Negative}, for any $T>0$, 
\begin{align} \label{eq_concentration1}
    \Pr\bc{\vec N_x - \vec \mu_x \SIGMA_x\geq T \Big \vert \cE_x} \leq \exp\bc{- \frac{T^2}{2 \vec v_x^2 + 2T\frac{s}{\Delta}\sqrt{\frac{m}{k}}/3}},
\end{align} 
as well as
\begin{align} \label{eq_concentration2}
\Pr\bc{\vec N_x - \vec \mu_x \SIGMA_x\leq - T \Big \vert \cE_x} \leq \exp\bc{- \frac{T^2}{2 \vec v_x^2 + 2T\frac{s}{\Delta}\sqrt{\frac{m}{k}}/3}}.
\end{align}
Since $\Pr\bc{\vec N_x \geq T_{\alpha} \big \vert \SIGMA_x=\tau} = \Erw\brk{\Pr\bc{\vec N_x \geq T_{\alpha} \big \vert \cE_x} \big \vert \SIGMA_x=\tau}$ for $\tau \in \{0,1\}$, \eqref{eq_concentration1} and \eqref{eq_concentration2} immediately yield
\begin{align}\label{eq_conc_3}
    \Pr\bc{\vec N_x \geq T_{\alpha} \big \vert \SIGMA_x=0} \leq \Erw\brk{\exp\bc{- \frac{T_\alpha^2}{2 \vec v_x^2 + 2T_\alpha \frac{s}{\Delta}\sqrt{\frac{m}{k}}/3}} \Bigg \vert \SIGMA_x=0} 
\end{align}
and
\begin{align} \label{eq_conc_4}
   & \Pr\bc{\vec N_x \leq T_{\alpha} \big \vert \SIGMA_x=1}  
   \leq \Erw\brk{\exp\bc{- \frac{\bc{T_{\alpha} - \vec \mu_x}^2}{2 \vec v_x^2 + \frac{2}{3}\bc{T_{\alpha} - \vec \mu_x}\sqrt{\frac{s^2 m}{\Delta^2 k}}}} \Bigg \vert \SIGMA_x=1}. 
\end{align}
To further bound the right hand sides of \eqref{eq_conc_3} and \eqref{eq_conc_4}, we next upper bound $\vec v_x^2$. Since each $\vec X_{x,a,t}^{(r:j)}$ given $\cE_x$ is a Bernoulli random variable with success probability $\vec p_x^{(r:j)}$, we have $ \Var(\vec X_{x,a,t}^{(r:j)}\vert \cE_x) \leq  \vec p_x^{(r:j)}$ and therefore
\begin{align*}
    \vec v_x^2 
    &\leq  \sum_{j=0}^{s-1} \vec \Delta_x[i+j] \sum_{r=0}^{j}\frac{1}{\frac{\Delta^2}{s^2}(j+1)^2\frac{k}{m}}\cdot\frac{\Gamma}{s} \cdot \frac{\vec k[i+r]}{|V[i+r]|}.
    \end{align*}
To deterministically bound the last expression, we define the event of having obtained a ``typical'' pooling scheme: Set
\begin{align*}
    \mathfrak P  &:=  \bigcap_{j=0}^{s-1} \cbc{\abs{\vec \Delta_x[i+j] - \frac{\Delta}{s}} \leq 2  \sqrt{\frac{\Delta}{s}} \log{n}} \cap \bigcap_{j=1}^{\ell+s-1} \cbc{\abs{\vec k[j] - \frac{k}{\ell}} \leq 2 \sqrt{\frac{k}{\ell}}\log{n}}.
\end{align*}
Due to Claims \ref{cor_conc_delta} and \ref{cor_conc_k}, the probability of $\mathfrak P$ is $1-o(n^{-2})$. On the other hand, on $\mathfrak P$, thanks to the exact identity $\Gamma m = \Delta n$,
\begin{align}\label{var_est}
    \vec v_x^2 &\leq \frac{\Gamma}{n} \cdot \frac{m}{\Delta} (1+o(1)) \cdot  \sum_{j=1}^s \frac{1}{j} =  (1+o(1)) \cdot  \sum_{j=1}^s \frac{1}{j} 
    \leq (1+o(1)) \ln s.
    \end{align}
  Substituting in (\ref{var_est}) into (\ref{eq_conc_3}) and observing that $\frac{2}{3}T_{\alpha}\frac{s}{\Delta}\sqrt{\frac{m}{k}}=o(1)$, we obtain the type I error bound from the statement of the theorem:
 \begin{align}\label{eq_conc_5}
    \Pr\bc{\vec N_x \geq T_{\alpha} \big \vert \SIGMA_x=0} \leq s^{-(1+o(1)) \frac{\alpha^2}{2}\cdot \frac{m}{k}} + o(n^{-2}).
\end{align} 
Turning to the type II error bound, observe that on $\mathfrak P$,
\begin{align}\label{dev_est}
  \bc{T_{\alpha} - \vec \mu_x}^2  &= (1+o(1)) (1-\alpha)^2 \cdot \frac{m}{k}\ln^2 s.
\end{align}
Substituting in (\ref{dev_est}) into (\ref{eq_conc_4}), we obtain
\begin{equation*}
    \Pr\bc{\vec N_x \leq T_\alpha \big \vert \SIGMA_x=1}   \leq s^{-(1+o(1))\frac{(1-\alpha)^2}{2}\cdot \frac{m}{k}}+ o\bc{n^{-2}}. \qedhere
\end{equation*}
\end{proof}

\subsection{Proof of Theorem \ref{thm_main_statement}} \label{SSec:Proof}

Given $\theta \in (0,1)$ and $\delta>0$, let $\vec G$ be the spatially coupled pooling scheme from Section \ref{SSec:SpCoup}, together with the $\ell$ additional pools $A_i$, $i=s, \ldots, \ell+s-1$. Set $\delta' = \delta/2$. We choose the parameters from \eqref{Eq:LS} to be
\begin{align} \label{Eq:final_par1}
    \eps_1= \theta\bc{\frac{3\delta'}{8(1+\delta')}}, \qquad 
    \eps_2 = \frac{\delta'}{4(1+\delta')} & \qquad \text{ and } \qquad  \eps_3=\frac{\delta'}{2(1+\delta')}
\end{align}
as well as a threshold parameter
\begin{align}\label{Eq:final_par2}
    \alpha = \frac{1}{1+\sqrt{\theta}}.
\end{align}
Notice that this choice of $\eps_1, \eps_2, \eps_3$ satisfies \eqref{Eq:Cond1}, \eqref{Eq:CondA} and \eqref{Eq:CondB}. 
Moreover, we fix the number $m$ of tests within the compartments to 
\begin{align*}
     2(1 + \delta')\cdot \frac{1+\sqrt{\theta}}{1-\sqrt{\theta}} \cdot \frac{1-\theta}{\theta} k = (1+\delta') \cdot c_{\theta}^{\star} \cdot \mdjackov^{\mathrm{inf}}.
\end{align*}
Then, since the number of additional tests $A_i, i \in [\ell+s-1]$, is $o(k)$, the total number of tests employed in the pooling scheme is bounded above by $(1+\delta) \cdot c_{\theta}^{\star} \cdot \mdjackov^{\mathrm{inf}}$. 
Let us define the bad event that there exists an individual whose exact neighbourhood sum does not correspond to its true state:
\begin{align}
    \mathfrak{B} &= \bigcup_{x \in V_{\mathrm{bulk}}} \left\{ \SIGMA_x=0, \vec N_x \geq T_{\alpha} \right\} \cup \left\{\SIGMA_x=1, \vec N_x \leq T_{\alpha} \right\}. 
\end{align}
We will first show that $\Pr\bc{\mathfrak{B}} = o(1)$. By the union bound,
\begin{align} \label{Eq:union_bound}
    \Pr\bc{\mathfrak{B}} 
    \leq   n \cdot \Pr\bc{\vec N_x \geq T_\alpha \big \vert \SIGMA_x=0} + k \cdot \Pr\bc{\vec N_x \leq T_\alpha \big \vert \SIGMA_x=1}.
\end{align}
Next, by Lemma \ref{Lem_conc},
\begin{align*}
     n \cdot \Pr\bc{\vec N_x \geq T_\alpha \big \vert \SIGMA_x=0}  &\leq  n \cdot  s^{ -(1+o(1))\frac{2 c_\theta^\star \alpha^2 (1-\theta) \bc{1+\delta'} }{2 \theta  }} 
     = n^{1 - (1-\eps_3)(1+\delta') + o(1)} = o(1)
\end{align*}
by our choice of $\eps_3$. 
Secondly,
\begin{align*}
     k \cdot \Pr\bc{\vec N_x \leq T_\alpha \big \vert \SIGMA_x=1} &\leq n^{\theta} \cdot s^{-(1+o(1)) \frac{2 c_{\theta}^{\star} (1-\alpha)^2 (1-\theta) \bc{1+\delta'}}{2 \theta  }}
     = n^{\theta\bc{1- (1-\eps_3)(1+\delta')}+o(1)} = o(1),
\end{align*}
again by the choice of $\eps_3$.
 Therefore, $\Pr\bc{\mathfrak{B}} = o(1)$, and it is sufficient to restrict to $\fB^c$ in the following.   

On $\fB^c$, for all $x \in V_{\mathrm{bulk}}$, the non-normalised neighbourhood sum $\vec U_x^j$ of $x$ agrees with its algorithmic estimate $\cU_x^j$ at the time when the label of $x$ is about to be inferred. We will now argue that for $n$ sufficiently large, also $\vec N_x = \cN_x$. The only difference between $\vec N_x^j$ and $ \cN_x^j$, $j=0, \ldots, s-1$, arises from the different normalisation. Indeed, thanks to Corollary \ref{cor_cond_exp_pd},
\begin{align*}
    \mathcal M_x^j - \bc{\Erw\brk{\vec U_x^j \vert \cE_x} -  \vec \Delta_x[j]\SIGMA_x} &= \vec \Delta_x[j] \bc{\frac{\Gamma}{s} -1 } \frac{ \SIGMA_x}{|V[i]| -1},
\end{align*}
where we again assume that $x \in V[i]$. Thus
\begin{align*}
    \abs{ \vec N_x - \cN_x} = \sum_{j=0}^{s-1} \bc{ \vec N_x^j - \cN_x^j} = \sum_{j=0}^{s-1} \frac{\vec \Delta_x[j] \bc{\frac{\Gamma}{s} -1 } \frac{ \SIGMA_x}{|V[i]| -1}}{\frac{\Delta}{s}(j+1)\sqrt{\frac{k}{m}}}.
    \end{align*}
To further bound the last expression, set
\begin{align*}
    \mathfrak D  &:=  \bigcap_{i=s}^{\ell+s-1} \bigcap_{x \in V[i]}\bigcap_{j=0}^{s-1} \cbc{\abs{\vec \Delta_x[i+j] - \frac{\Delta}{s}} \leq 2 \ln n \sqrt{\frac{\Delta}{s}}}.
\end{align*}
On $\fB^c \cap \fD$, which is a w.h.p. event due to our initial computation and Claim \ref{cor_conc_delta}, we thus have
\begin{align*}
    \abs{ \vec N_x - \cN_x} = \sum_{j=0}^{s-1} \bc{ \vec N_x^j - \cN_x^j} = O\bc{\frac{\ell \Gamma}{n s}} = O\bc{n^{\theta(\eps_3-\eps_2)-\eps_1}} = o(1),
    \end{align*}
since our parameters are chosen to obey condition \eqref{Eq:CondA}. 
As the threshold $T_{\alpha}$ is of order $\ln s = \omega(1)$, on $\fB^c \cap \fD$, for large enough $n$, classification of every item according to $\cN_x$ will lead to the same classification as using $\vec N_x$, and therefore to correct classification. Consequently, w.h.p., \algoname\hspace{0.01cm} classifies all items correctly. 

Finally, we briefly argue why \algoname\hspace{0.01cm} is efficient. First, creating the artificial items along with their labels takes time $O\bc{(s-1)\lceil n/\ell\rceil}$. Second, partitioning all items into compartments takes time $O(n)$, and setting up the $\ell$ additional pools $A_i$ and evaluating their test results takes time $O(\ell n)$. Next, sampling the neighbourhood for all spatially coupled pools $a_i$ and evaluating their test outcomes takes time $O(mn)$, where the former can be done by labelling each item by an independent uniform random variable on $[0,1]$ associated to test $a_i$ and then assigning $a_i$ the $\Gamma/s$ items with lowest label in each possible compartment. Partitioning the tests into compartments takes time $O(m)$. This concludes setting up the pooling scheme. Finally, computation of all neighbourhood sums $\cN_x$ uses about $O(nm\Gamma s/\ell) = O(n^3)$ elementary operations. Therefore, both setting up the pooling scheme and running the thresholding algorithm is clearly doable in polynomial time. 
\qed

\subsubsection{Parameter Choice in the Pooling and Thresholding Schemes \eqref{Eq:final_par1}-\eqref{Eq:final_par2}} \label{SSSec:Opt}
To prove that our general combination of the spatial coupling with a thresholding approach classifies all items correctly, we employ a union bound over all items within a compartment, and then over all compartments (see \eqref{Eq:union_bound}). This bound solely depends on the concentration inequality from Lemma \ref{Lem_conc}, which is formulated in terms of general parameter choices, and we have therefore aimed to heuristically choose the parameters in a way that minimises $c_\theta$.

Specifically, to ensure correct classification of all items w.h.p., the exponents in \eqref{Eq:union_bound} need to be negative. From this, we obtain the two necessary conditions
\begin{align}
    1- c_\theta \alpha^2 (1-\theta)   (1+\delta)(1-\eps_3)& <0, \label{eq:inErr1a} \\
    \theta- c_\theta (1-\alpha)^2 (1-\theta)  (1+\delta)(1-\eps_3)& <0. \label{eq:inErr2a}
\end{align}
As pointed out, our aim now is to minimise $c_\theta$ subject to (\ref{eq:inErr1a}) and (\ref{eq:inErr2a}). We observe that both inequalities can be satisfied by choosing $\eps_3$ in dependence on $\delta$ small enough, if we can find $\alpha$ and $c_\theta$ that simultaneously satisfy
\begin{align}
    c_\theta \alpha^2 (1-\theta) & =1, \label{eq:inErr1a2} \\
   c_\theta (1-\alpha)^2 (1-\theta) & =\theta. \label{eq:inErr2a2}
\end{align}
Solving \eqref{eq:inErr1a2} and \eqref{eq:inErr2a2} for $c_{\theta}$ and $\alpha$ gives
    \begin{align*}
        \alpha_{\pm} =  & \frac{1\pm \sqrt{\theta}}{1-\theta}, \qquad c_{\theta, \pm} = \frac{1}{2\alpha-1}. 
    \end{align*}
Since $\alpha_+ >1$ for all $\theta \in (0,1)$, we obtain
\[ \alpha_{\mathrm{opt}} = \frac{1-\sqrt{\theta}}{1+\theta} = \frac{1}{1+\sqrt{\theta}}, \qquad c_{\theta, \mathrm{opt}} =\frac{1+\sqrt{\theta}}{1-\sqrt{\theta}}.\]
For this specification of $\alpha_{\mathrm{opt}}$, given $\delta$, we can now choose $\eps_3$ sufficiently small to ensure that both (\ref{eq:inErr1a}) and (\ref{eq:inErr2a}) are satisfied. 

What remains to be done is to specify choices of $\eps_1,\eps_2, \eps_3$ that satisfy the four conditions (\ref{Eq:Cond1}), (\ref{Eq:CondA}), (\ref{Eq:CondB}) as well as  
\begin{align}
    \eps_3 < \frac{\delta}{1+\delta}, \label{Eq:Cond3a}
\end{align}
which we obtain from the condition $(1+\delta)(1-\eps_3)>1$ in the optimisation of $c_\theta$. There are many choices for these parameters, and it is straightforward to check that the choice specified in the proof of Theorem \ref{thm_main_statement} satisfies all four conditions.

As a final remark, we would like to point out that even allowing for $o(k/\ell)$ errors per compartment, and employing a second stage to correct these errors, does not appear to improve the constant $c_{\theta}$.

\section{Restricting to $\mathrm{d}=1$: Proof of Corollary \ref{thm_main_corollary}} \label{sec:d1}

We finally argue why it is sufficient to prove Theorem \ref{thm_main_statement} for the case $d=1$.
Suppose that for all $\delta, \theta \in (0,1)$, 
Theorem \ref{thm_main_statement} holds true. 

Indeed, let $d \geq 2, \delta >0$ and $\theta = \max\{\theta_i: i \in [d]\}$ as before be the prevalence of the dominant among the non-zero labels. Moreover, let $D$ be the number of $i \in [d]$ such that $\theta_i = \theta$. We then choose the spatially coupled pooling scheme for $d=1$ on $(1+\delta/2)2\frac{1-\theta}{\theta} \cdot \frac{1+\sqrt{\theta}}{1-\sqrt{\theta}}n^{\theta}$ pools from Theorem \ref{thm_main_statement}, along with $\sum_{i: \theta_i < \theta}(1+\delta)2\frac{1-\theta_i}{\theta_i} \cdot \frac{1+\sqrt{\theta_i}}{1-\sqrt{\theta_i}} k_i = o(n^{\theta})$ independent spatially coupled pooling schemes as in the proof of Theorem \ref{thm_main_statement}, each one corresponding to one of the parameters $\theta_i$ such that $\theta_i < \theta$. For $n$ sufficiently large, this yields less than $(1+\delta)2\frac{1-\theta}{\theta} \cdot \frac{1+\sqrt{\theta}}{1-\sqrt{\theta}}n^{\theta}$ pools in total.

We then proceed in $d$ rounds to infer one label after the other. 
In the first round, choose an arbitrary label $o$ with $\theta_o = \theta$. We then infer all items of label $i$ as follows. In any given item compartment $V[i]$, for $x \in V[i]$, \algoname\hspace{0.01cm} will be able to use all the histogram information obtained from the pools that $x$ is part of. \algoname\hspace{0.01cm}  will now, in contrast to the true histogram information, treat all items of label $p \notin \{0,o\}$ as having label $0$, and then compute the unexplained neighbourhood sums for $x$ based on this modified histogram with two labels, exactly as in the $d=1$ case. Also the normalisation uses exactly the same parameters as in the $d=1$ case.
This will lead to correct identification of all items with label $o$ w.h.p. 

If there is a second label $o'$ with $\theta_{o'} = \theta$, we then use the same pooling scheme and do a second round of thresholding. This time, all items of label $p \notin \{0,o'\}$ will be treated as having label $0$ in the computation of the unexplained neighbourhood sums. Again, the normalisation uses exactly the same parameters as in the $d=1$ case.
This will lead to correct identification of all items with label $o'$ w.h.p. Moreover, the use of the same pooling scheme in the various rounds does not pose a problem, since each round does not take into account the results from previous stages.

We do this until all items with a label $o$ such that $\theta_o = \theta$ are identified. For the remaining ones, we proceed exactly as before, but with a different pooling scheme as in Theorem \ref{thm_main_statement} as before. This is possible, as it asymptotically does not change the leading order of the number of pools. \qed

\section*{Acknowledgements}
We thank an unknown reviewer of a preliminary version for an excellent idea on how to simplify the seed.

\section*{Funding}
\noindent The first author was supported by DFG FOR 2975. The second and thirds author are supported in part by the NWO Gravitation project NETWORKS under grant no. 024.002.003. The third author was also supported by ERC-Grant 772606-PTRCSP.
{\includegraphics[scale=0.1]{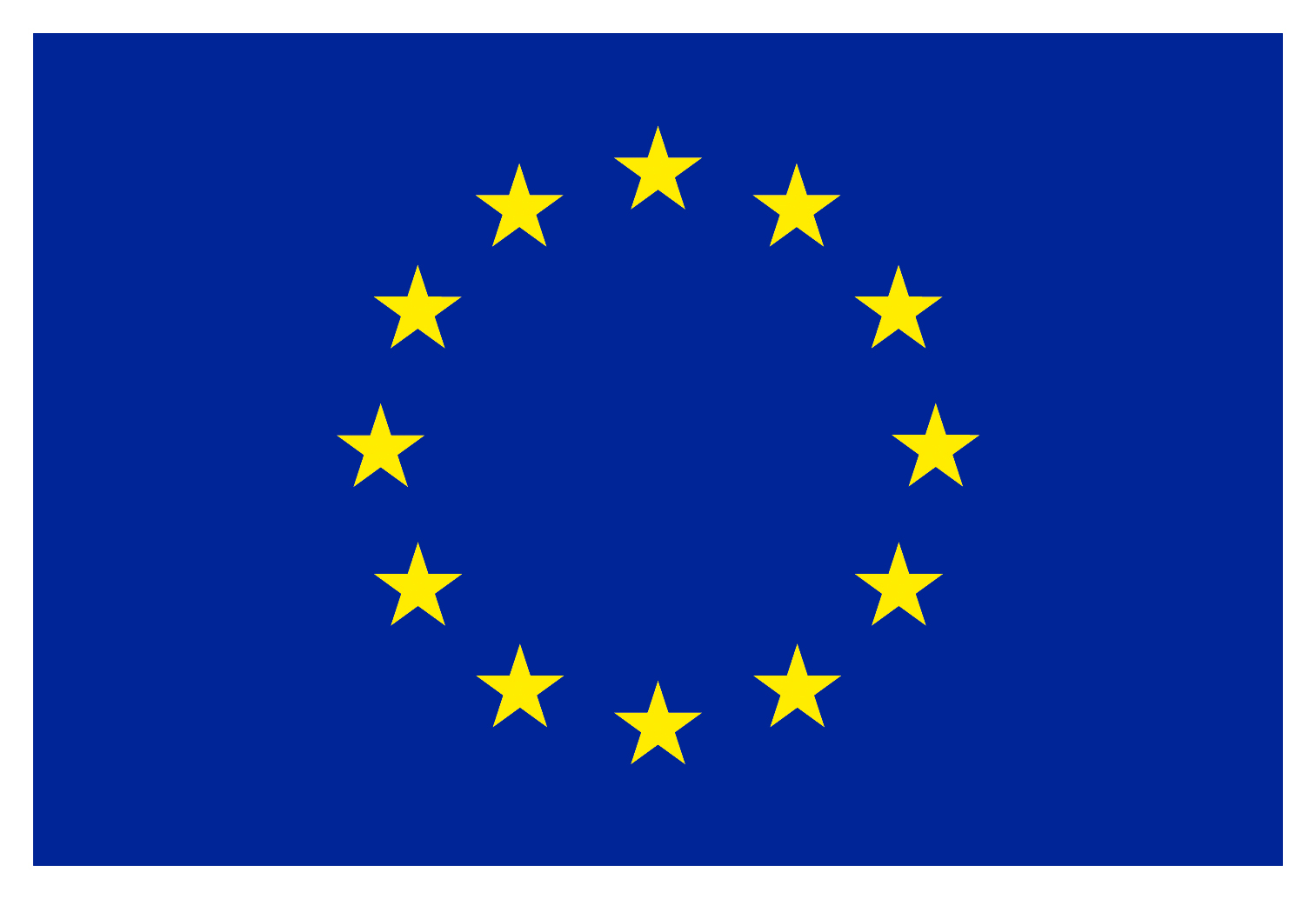}}


\printbibliography

\begin{appendix}
\section{Notation} \label{appendix_notation}
\vspace{-10pt}
\renewcommand{\arraystretch}{0.4}
\begin{table}[H]
    \centering
    \begin{tabularx}{\linewidth}{c X}
        \hline
         \textbf{Variable} &\textbf{Description} \\
         \hline
         $n$ & Number of items with unknown labels. \\ \hline 
         $k$ & Number of items with label $1$. \\ \hline
         $\theta$ & Growth parameter of the number of items with label $1$: $k=n^{\theta}$, $0<\theta<1$. \\ \hline    
         $m$ & Number of tests / measurements. We use $m = 2 c_{\theta} \bc{1 + \delta} (1-\theta)/\theta \cdot k$ throughout. \\ \hline
         $\mdjackov^{\mathrm{inf}}$ & Information-theoretically optimal number of tests / measurements: $\mdjackov^{\mathrm{inf}} = 2 (1-\theta)/\theta \cdot k$. \\ \hline
         $c_{\theta}$ & Number of tests divided by the information-theoretically optimal number of tests times a small buffer: $c_{\theta} = m /[(1+\delta)\mdjackov^{\mathrm{inf}}]$.
         \\ \hline
         $x_i$& Item $i$, $i = 1, \dots, n $.\\ \hline
         $V_{\text{bulk}}$ & Set of the $n$ items with unknown labels: $V_{\text{bulk}}=\{x_1, \ldots, x_n\}$. \\ \hline
         $a_i$& Test / measurement $i$, $i = 1, \dots, m$.\\ \hline
         $\SIGMA_i$& Label of item $x_i$. \\ \hline
         $\SIGMA$& Vector of item labels: $\SIGMA = (\SIGMA_1, \ldots, \SIGMA_n)$ (unknown to the algorithm). \\ \hline
         $\hat \SIGMA_i$ & Result of test $a_i$. \\ \hline
         $\hat \SIGMA$ & Vector of test results: $\hat \SIGMA= (\hat\SIGMA_1, \ldots, \hat\SIGMA_m)$. \\ \hline
         $\vec G$& SC pool design, represented by a random bipartite graph on the sets of items and pools. \\ \hline
         $\ell$ & Number of compartments consisting of items with unknown labels in the SC pool design.\\ \hline 
         $s$ & Number of compartments consisting of items with known labels in the SC pool design, plus one (sliding window).  \\ \hline
         $n'$& Number of auxiliary items. We set $n' = (s-1)\ceil{ n/\ell }$.\\ \hline 
         $\SIGMA'$& Vector of labels of the $n'$ auxiliary items 
        (known to the algorithm). \\ \hline
        $V_{\text{seed}}$ & Set of the $n'$ auxiliary items with known labels. \\ \hline\hline

    \end{tabularx}

    \label{tab:my_label}
    \caption{Summary of the main notation. Above, we abbreviate `spatial coupling' to `SC'.}
    \end{table}

    \begin{table}[!h]
    \centering
    \begin{tabularx}{\linewidth}{c X}
        \hline
         \textbf{Variable} &\textbf{Description} \\
         \hline
         

         $\tilde \sigma$ & Algorithmic estimate of $\SIGMA$, initialised by the all-zero vector. \\ \hline
         $V[i]$& Item compartment $i$ in the SC pool design, $i = 1, \ldots, \ell+s-1$.
         \\ \hline
         $F[i]$& Pool compartment $i$ in the SC pool design, $i = 1, \ldots, \ell+s-1$. 
         \\ \hline
         $\vec k[i]$& Number of items with label one in item compartment $V[i]$, $i=1, \ldots, \ell-s-1$. 
         \\ \hline
          $A_i$& Pool containing all items of $V[i]$, $i=s,\dots,\ell+s-1$. 
         \\ \hline
         $\partial x$& Set of pools in the SC pool design that item $x$ is part of. 
         \\ \hline
         $\partial a$& Set of items in pool $a$ in the SC pool design. \\ \hline
         $\vec \Delta_x[j]$ & The number of pools in $F[i+j-1$] that item $x \in V[i]$ is part of: 
         $\vec \Delta_x[j] = |\partial x \cap F[i+j-1]|$.\\ \hline
         $\vec \Delta_x$ & The number of pools in $F[i] \cup \ldots \cup F[i+s-1$] that item $x \in V[i]$ is part of: $\vec \Delta_x = \sum_{j=0}^{s-1}\vec \Delta_x[j]$. 
         \\ \hline
         $\Delta$ & Expected number of pools in $F[i] \cup \ldots \cup F[i+s-1$] that item $x \in V[i]$ is part of: $\Delta = \Erw[\vec \Delta_x]$. \\ \hline
         $\Gamma$ &  Deterministic number of (distinct) items in each of the tests $a_i$, $i=1, \ldots, m$. \\ \hline
         $\vec U_x^j$ & Contribution of compartment $F[i+j]$ to the unexplained neighbourhood sum of item $x \in V[i]$. \\ \hline
         $\cU_x^j$ & Estimated contribution of compartment $F[i+j]$ to the unexplained neighbourhood sum of item $x \in V[i]$, based on the algorithmic estimate $\tilde \sigma$. \\ \hline
         $\vec N_x^j$ & Contribution of compartment $F[i+j]$ to the unexplained normalised and centered neighbourhood sum of item $x \in V[i]$.
         \\ \hline
         $\cN_x^j$ & Estimated contribution of compartment $F[i+j]$ to the unexplained normalised and centered neighbourhood sum of item $x \in V[i]$, based on the algorithmic estimate $\tilde \sigma$.
         \\ \hline
         $\vec N_x$& Unexplained normalised and centered neighbourhood sum of item $x$ .
         \\ \hline
         $\cN_x$& Estimated unexplained normalised and centered neighbourhood sum of item $x$, based on the algorithmic estimate $\tilde \sigma$. 
         \\ \hline
         $T_{\alpha}$ & Threshold that the algorithm uses to distinguish between $\SIGMA_x=0$ and $\SIGMA_x=1$: $T_{\alpha} \approx \alpha \cdot \Erw\brk{\vec N_x \vert \SIGMA_x=1}$, $\alpha \in (0,1)$.
         \\ \hline \hline
    \end{tabularx}
    
    \caption{Continued summary of the main notation. Above, we abbreviate `spatial coupling' to `SC'.}
\end{table}

\section{Concentration Inequalities} \label{Appendix_inequalities}


\begin{theorem}[Chernoff bound for the binomial distribution \cite{janson2011random}] 
\label{lem_chernoff}
Let $\vX \sim \Bin \bc{n, p}$. Then for any $\eps>0$, 
\begin{align*}
    & \Pr\bc{ \vX \geq (1 + \eps) \Erw \brk{\vX}} \leq \exp\bc{-\frac{\eps^2}{2 + 2\eps/3} \Erw\brk{\vX}} \qquad \text{and}  \\
    & \Pr\bc{\vX \leq (1 - \eps) \Erw\brk{\vX}} \leq \exp\bc{-\frac{\eps^2}{2} \Erw\brk{\vX}}.
\end{align*}
\end{theorem}

\begin{theorem}[Chernoff bound for the hypergeometric distribution \cite{Janson99onconcentration}]\label{Hyp}
Let $\vX$ be a Hyp$(N,M,K)$-distributed random variable and let $t \geq 0$. Then
\begin{align*}
    & \Pr\bc{\vX - \Erw[\vX] \geq t} \leq \exp\bc{-\frac{t^2}{2(KM/N+t/3)}} \qquad \text{and} \\ & \Pr\bc{\vX - \Erw[\vX] \leq -t} \leq \exp\bc{-\frac{t^2}{2KM/N}}.
\end{align*}
\end{theorem}

\vspace{-20pt}
\begin{theorem}[Bernstein inequality \cite{Chen2019}]\label{Bernstein}
Let $\vX_1, \ldots, \vX_n$ be independent random variables such that $\Erw\brk{\vX_i} = 0$ and $\abs{\vX_i} \leq z$ almost surely for all $i \in [n]$ and a constant $z>0$. Moreover, let $\sigma^2:=\frac{1}{n} \sum_{i=1}^n\Var\bc{\vX_i}$. Then for all $\eps > 0$, 
\begin{align*}
    \Pr\brk{\sum_{i=1}^n \vX_i \geq \eps n} \leq \mathrm{exp}\bc{-\frac{n \eps^2}{2\sigma^2 + 2z\eps/3}}.
\end{align*}
\end{theorem}

\vspace{-10pt}
\begin{definition}[Negative association] \label{def_NA} Let $\vX = (\vX_1, \ldots, \vX_d)$ be a random vector. Then the family of random variables $\vX_1, \ldots, \vX_d$ is said to be negatively associated if for every two disjoint index sets $I,J \subseteq [d]$ we have
\[ \Erw\brk{f(\vX_i: i \in I) g(\vX_j: j \in J)} \leq \Erw\brk{f(\vX_i: i \in I)} \Erw\brk{ g(\vX_j: j \in J)} \]
for all functions $f:\RR^{|I|} \to \RR$ and $g:\RR^{|J|} \to \mathbb{R}$ that are either both non-decreasing or both non-increasing.
\end{definition}
\vspace{-15pt}
\begin{theorem}[Corollary EC.4 of \cite{Chen2019}] \label{Bernstein_Negative}
Bernstein’s inequality (\Prop~\ref{Bernstein}) holds for negatively associated random variables.
\end{theorem}

\end{appendix}

\end{document}